\documentclass[11 pt]{amsart}
\usepackage[latin1]{inputenc} 
\usepackage[pctex32]{graphics}
\usepackage{amsfonts}
\usepackage{amssymb,amscd,latexsym}
\usepackage{amsmath}
\usepackage{epsfig}
\usepackage{color}
\usepackage[all]{xy}
\usepackage{hyperref}
\usepackage{tikz}
\newtheorem{thm}{Theorem}[section]

\newtheorem{prop}[thm]{Proposition}
\newtheorem{defin}[thm]{Definition}
\newtheorem{lema}[thm]{Lemma}
\newtheorem{rmk}[thm]{Remark}

\newtheorem{ex}[thm]{Example}

\def\OX{\mathcal{O}}
\def\Ee{\mathcal{E}}

\def\T{\mathcal{T}}

\def\K{\mathbb{K}}

\def\C{\mathbb{C}}

\def\grass{\mathbb{G}}

\def\P{\mathbb{P}}

\def\Hess{\operatorname{Hess}}
\def\hess{\operatorname{hess}}
\def\codim{\operatorname{codim}}
\def\reg{\operatorname{reg}}

\def\Gor{\operatorname{Gor}}
\def\Ann{\operatorname{Ann}}
\def\Hilb{\operatorname{Hilb}}
\def\rk{\operatorname{rk}}
\def\Sing{\operatorname{Sing}}

\begin{document}
\title{On the Lefschetz locus in $\Gor(1,n,n,1)$}
\author[Bezerra]{Lenin Bezerra}
\address{\sc Lenin Bezerra\\
Universidade Federal de Pernambuco, Brasil}
\email{lenin.bezerra@ufpe.br}

\author[Ferrer]{Viviana Ferrer}
\address{\sc Viviana Ferrer\\
Universidade Federal Fluminense, Rua Alexandre Moura 8 - S\~ao Domingos, 24210-200 Niter\'oi, Rio de Janeiro, Brasil}
\email{vferrer@id.uff.br}

\author[Gondim]{Rodrigo Gondim}
\address{\sc Rodrigo Gondim\\
Universidade Federal Rural de Pernambuco, av. Dom Manoel de Medeiros s/n, Dois Irm\~aos, 52171-900 Recife, Pernambuco, Brasil}
\email{rodrigo.gondim@ufrpe.br}

\thanks{
The first author was supported by CAPES. 
The last author was partially supported by CNPQ. Ferrer acknowledges support from CNPq (Grant number 408687/2023-1).}

%


\begin{abstract}
We study two special families of cubic hypersurfaces with vanishing hessian in $\mathbb{P}^N$, obtaining rational parametrizations and computing their degree in $\mathbb{P}(S_3)$. For $N \leq 6$, these two families exhaust the locus of cubics with vanishing hessian that are not cones. As a consequence, via Macaulay-Matlis duality, we obtain a description of the locus in $\mathrm{Gor}(1, n, n, 1)$ corresponding to those algebras that satisfy the Strong Lefschetz property, for $n \leq 7$.



\end{abstract}

\maketitle


\section{Introduction}
Artinian Gorenstein $\K$-algebras serve as models for cohomology rings with coefficients over a field $\K$ in various categories. Over the past two decades, interest from an algebraic perspective has grown, making this a fertile area of research. It is quite surprising that even in this more abstract context, the Lefschetz properties maintain a profound connection with other areas of mathematics.

Originally, S. Lefschetz proved his celebrated Hard theorem in the category of complex smooth projective varieties. Nowadays, there are versions of his theorem in the category of differential geometry, vector bundle theory, matroid theory, among others. The study of Lefschetz's Theorem and Hodge theory is until these days very influential in the understanding of the topology of projective algebraic varieties.

From the algebraic viewpoint, let $A = \bigoplus_{k=0}^{d} A_{k}$ be a standard graded Artinian $\K$-algebra, i.e. a graded $\K$-algebras with a finite number of graded components. We define the socle degree of $A$ to be  $d$ if $A_d\neq 0$. The Hilbert function of $A$ can be described as a vector $\Hilb(A) = (1, h_1, \ldots, h_d)$, where $h_k = \dim_{\K}A_{k}$.  The codimension of $A$ is $h_1$, which coincides with the embedding dimension.

We say that an Artinian $\K$-algebra has Poincar\'e duality if $\dim_{\K}A_{0} = \dim_{\K}A_{d} = 1$ and, for every $k = 1, \ldots, d-1$, the multiplication map $$A_k \times A_{d-k} \rightarrow A_{d}$$ defines  a perfect pairing. Such  algebras are precisely the Artinian Gorenstein (AG) algebras. Recall that Poincar\'e duality implies that $\dim_{\K}A_{k} = \dim_{\K}A_{d-k}$. By Macaulay-Matlis duality (see \cite{MW}), Artinian Gorenstein algebras admit,  in characteristic zero, a differential presentation $$A = Q/\Ann_{f},$$ where $f \in S = \K[x_1, \ldots, x_{n}]$ is a homogeneous polynomial of degree $d$, $Q = \K[X_1, \ldots, X_{n}]$ is the polynomial ring of differential operators acting on $S$ via $X_{i}(x_{j}) = \delta_{ij}$, and $\Ann_{f} = \{\alpha \in Q; \ \alpha(f) =0\} \subset Q$ is a homogeneous ideal. 

An abstract version of the Hard Lefschetz theorem was introduced by R. Stanley in \cite{St}, and is known as the Strong Lefschetz Property (SLP). We say that an AG algebra $A$ has the SLP if there exists an element $L \in A_{1}$ such that, for every $k \in \{1, \ldots, \lfloor \frac{d}{2} \rfloor \}$, the multiplication map  $$\cdot L^{d-2k}: A_{k} \rightarrow A_{d-k}$$ is an isomorphism.

Fixing a Hilbert function $T = (1, h_1, \ldots, h_{d-1}, 1)$, Iarrobino and Kanev, in a series of papers compiled in \cite{IK}, described the parameter space $\Gor(T)$ of Artinian Gorenstein algebras having $T$ as Hilbert function. In the very initial case where $T = (1, n, n, 1)$, $\Gor(T)$ parametrizes Artinian Gorenstein algebras of socle degree $3$ and codimension $n$. By Macaulay-Matlis duality, these algebras admit a presentation $$A = \K[X_{1}, \ldots, X_{n}]/ \Ann_f,$$ where $f \in \K[x_{1}, \ldots, x_{n}]_{3}$ is a homogeneous cubic polynomial whose derivatives are linearly independent. In \cite{MW}, J. Watanabe and T. Maeno provided  a criterion for SLP in terms of higher-order Hessians, which,  in  the particular case of socle degree $3$, states that the algebra $A = Q/ \Ann_{f}$ has the SLP if and only if the determinant of the Hessian matrix of $f$ does not vanish identically, i.e. $\hess_{f} \neq 0$.

Let $X$ denote the hypersurface defined by $f$ on $ \P^{n-1}$. Saying that the partial derivatives of $f$ are linearly independent is equivalent to saying that $X$ is not a cone. Thus, to describe -the complement of- the Lefschetz locus in $\Gor(1, n, n, 1)$, we study cubic hypersurfaces  $V(f)$, that are not cones and satisfy $\hess_{f} = 0$. This is a very classical problem tracing back to Hesse, who claimed in \cite{hesse1} and \cite{hesse2} that for any degree $d \geq 3$, a hypersurface $X = V(f) \subset \P^{N}$ has vanishing Hessian if and only if $X$ is a cone. However, Hesse's claim is false, as Sylvester pointed out by providing a correct determinantal criterion to characterize hypersurfaces with vanishing Hessian. The key step is called the Gordan-Noether theory, which, in the language of algebraic geometry, clarifies the distinction between being a cone and having vanishing Hessian determinant, proves Hesse's claim in $\P^{2}$ and $\P^{3}$, and classifies all the counter-examples in $\P^{4}$. The chapter $7$ of \cite{Ru} is a good reference to introduce this topic. 

In the particular case of cubic hypersurfaces, U. Perazzo presented the minimal counterexample for Hesse's claim: $X = V(f) \subset \P^{4}$ given by \begin{equation}\label{exhess=0notcone}
    f = xu^2 + yuv + zv^{2}.
\end{equation}
It is easy to see that $f$ is not a cone but has vanishing hessian. 

The work of Perazzo was revisited in \cite{GRu}, where the authors gave an algebro-geometric classification of cubics with  vanishing hessian in $\P^{4}, \ \P^{5}$, and   $\P^{6}$, and discussed the problems that arise in $\mathbb{P}^7$ and higher dimensions.

Hypersurfaces with vanishing hessian are developable. Developable hypersurfaces are the counterparts, in algebraic geometry, of hypersurfaces with zero curvature in differential geometry. They are ruled by linear subspaces; that is, for any point $x \in X$, there exists a $k$-linear subspace $L_{x} \subset X$ such that $x \in L_{x}$ and, moreover, the tangent space $T_{y}X$ is constant for every $y \in L_{x}$.  In \cite{FFG}, the authors proved that if $X = V(f) \subset \P^{4}$ is an irreducible cubic hypersurface not a cone, then $X$ is developable if and only if it is projectively equivalent to a linear section of the secant variety of the Veronese surface. In this same paper, they used the classification of cubic hypersurfaces that are not cones and having vanishing hessian, to describe the Lefschetz locus in $\Gor(1, 5, 5, 1)$.

In this paper, we extend the results in \cite{FFG} by describing the Lefschetz locus in $\Gor(1,6,6,1)$ and $\Gor(1,7,7,1)$. We identify two families that exhaust the algebras failing SLP, which we call the maximal and the minimal families. 
For $n\leq 7$, these families describe the irreducible components of the reduced structure of the locus of algebras in $\Gor(1,n,n,1)$ failing SLP.  
For $n=8$, however, there exists an example of an algebra failing  SLP thas is neither a member of the maximal nor the minimal family.

We now describe the structure of the paper.

In the second section, we review earlier results on cubic hypersurfaces with vanishing hessian and introduce the two families of hypersurfaces that we aim to study.  

In the third section, we recall some previous results concerning Artinian Gorenstein algebras and Lefschetz properties. In particular, we make explicit the relation between algebras having the Strong Lefschetz Property (SLP) and cubics with vanishing Hessian. We also recall the notion of Jordan types and compute the Jordan types for the two families under consideration.

In the fourth section, we describe parameter spaces for the minimal and maximal families. We also describe the parameter space of the intersection of the minimal and maximal families with cones. Additionally, the tangent space to the scheme $(hess=0)\subset \P(S_3)$ is computed at a general point of these families, showing that the maximal family is reduced, while the other families are not. 

In the last section, we describe the Lefschetz locus in $\Gor(1,6,6,1)$ and $\Gor(1,7,7,1)$.

We include an appendix consisting of the Macaulay2 codes used to compute the degree of each family.

Since we work in both geometric and algebraic settings, throughout this paper we set $n=N+1$ in order to study cubic hypersurfaces $X=V(f) \subset \P^{N}$ with $f \in \K[x_0, \ldots, x_{N}]_3$,  and the associated Artinian Gorenstein $\K$-algebra $A$ whose Hilbert function is $\Hilb(A) = (1, n, n, 1)$, provided that $X$ is not a cone.

\section{Cubics with vanishing Hessian}
In this section we introduce the basic notations and results about cubic hypersurfaces with vanishing Hessian and their geometry.

\subsection{The polar map}
  
  We will work over an algebraically closed field of characteristic zero. 

\begin{defin}\rm
  Let $X \subset \P^N$ be an irreducible projective variety. The vertex of $X$ is the closed subset
$$\operatorname{Vert}(X)= \{ p \in X | \langle p,q \rangle \subset X,\  \forall q \in X\}$$
 where if $p,q\in X$,  $\langle p,q \rangle$ denotes the line join $p,q$.
  A projective variety $X \subset \P^N$ is a cone if $\operatorname{Vert}(X) \neq \emptyset$. In this case, 
$\operatorname{Vert}(X)$ is a linear subspace of $\P^N$.
\end{defin}

Cones have vanishing Hessian determinant, but the converse is not true in general, as shows the example \ref{exhess=0notcone} in the Introduction. 

\begin{defin}\rm\label{polar}
The polar map of a hypersurface $X=V(f) \subset \P^N$ is the rational map given by the derivatives of $f$.
$$\begin{array}{c} \Phi_f:  \P^N  \dashrightarrow   \check{\P}^{N} \\
                                                             \\
  \Phi_f(p)=(\frac{\partial f}{\partial x_0}(p), \frac{\partial
f}{\partial x_1}(p),  ..., \frac{\partial f}{\partial x_N}(p) ).
\end{array}$$
The polar image of $X$ is $Z = \overline{\Phi_{f}(\P^N)}$.
\end{defin}

The next proposition clarifies the difference between a hypersurface being a cone and the vanishing of its Hessian determinant.
\begin{prop}\label{prop.hess_vs_cone}
  Let $X = V(f) \subset \P^N$ be a hypersurface, $\Phi_f$ the associated polar map, and $Z = \overline{\Phi_f(\P^N)}$ the polar image.
  \begin{enumerate}
    \item $X$ has vanishing Hessian $\Leftrightarrow$ $Z \subsetneq \check{\P}^{N}$ $\Leftrightarrow$ The partial derivatives of $f$ are algebraically dependent.
    \item $X$ is a cone $\Leftrightarrow$ there is a hyperplane $H \subset \check{\P}^{N}$ such that $Z \subset H$  $\Leftrightarrow$ The partial derivatives of $f$ are linearly dependent.
  \end{enumerate}
\end{prop}
  
\begin{proof}
 See \cite{CRS}.
\end{proof}

The cubic (\ref{exhess=0notcone}) of the introduction appears  in the works of Gordan and Noether \cite{GN} and Perazzo \cite{Pe}, where it is called 
{\it un esempio semplicissimo}, as an example of a cubic with $\hess_f = 0$ that is not a cone.

\begin{ex} \rm
Let $X = V(f) \subset \P^4$ be the irreducible hypersurface given by 
$$f = x_0x_3^2+x_1x_3x_4+x_2x_4^4$$
An easy calculation shows that $X$ is not a cone. On the other hand, 
$f_0f_2=f_1^2$ is an algebraic relation among the partial derivatives of $f$, so $\hess_f = 0$.   
\end{ex}

The next result will be useful in the sequel. Its proof can be found in the original work of Perazzo, see \cite{Pe} in the cubic case, and for general 
degree in \cite[pg.21]{ZakHesse}. 

\begin{prop}\label{Zdual_lugar_singular}
  Let $X = V(f) \subset \P^N$ be a hypersurface with vanishing hessian, and $Z^*$ the dual of the polar image of $X$. Then
  $$Z^* \subset \Sing(X)_{red}.$$

\flushright{$\square$}

\end{prop}

\subsection{Perazzo hypersufaces}

A Perazzo cubic hypersurface is an irreducible $X=V(f) \subset \P^N$
that up to a projective transformation has a canonical form:
\[f=\displaystyle \sum_{i=0}^k x_ig_{i}+h. \]
Where $g_{i} \in \K[x_{N-m+1},\ldots, x_N]_2$ are linearly independent and algebraicaly dependent and $h\in \K[x_{k+1},\ldots,x_N]_3$ and $k+m\leq N$.

\subsubsection{The Perazzo map}

The notion of the Perazzo map was implicitly introduced in \cite{Pe}, see also \cite{GRu}. 

\begin{defin} \rm
Notation as in Definition \ref{polar}. Let $X = V(f) \subset \P^N$ be a reduced hypersurface with vanishing hessian.  {\it The Perazzo map of $X$} is the rational map:
$$\begin{array}{cccc} \mathcal{P}_X: & \P^N & \dashrightarrow & \mathbb{G}(\codim(Z),N+1) \\
\ & p & \mapsto & (T_{\Phi(p)}Z)^* \end{array}$$

defined in the open set $\mathcal{U} = \Phi^{-1}(Z_{\reg})$,
where $Z_{\reg}$ is the locus of smooth  points of $Z$.

The image of the Perazzo map will be denoted by $W_X=\overline{\mathcal{P}_X(\P^N)} \subset \mathbb{G}(\codim(Z),N+1)$, or simply by $W$,  and its dimension $\mu = \dim W$ is called the {\it Perazzo rank of $X$}. 
\end{defin}

 We are particularly interested in the case where $\codim(Z)=1$; and therefore, we shall assume this condition from now on. In this case, we have:


\[\begin{array}{ccccc}
\mathcal{P}_X:    \P^N  &  {\dashrightarrow} &   \P^N  
       \end{array}\]
and $W_X=Z^*$.

From \cite[Theorem 3.1]{GRu}, if $X$ is not a cone, then 
\begin{equation}\label{N>=}
    \mu = \dim (Z^*) \leq \frac{N-2}{2}.
\end{equation}

The general fiber of the Perazzo map is linear, of dimension $N-\mu$, see \cite[Theorem 2.5]{GRu}. 
Some hypersurfaces of interest have special Perazzo fibers, in the sense of the following definition.

\begin{defin}\rm
An irreducible cubic hypersurface $X\subset\P^N$ with vanishing hessian, not a cone, will be called a {\it Special Perazzo Cubic Hypersurface} if the general fibers of its Perazzo map determine a congruence of linear spaces passing through
a fixed $\P^{N-\mu-1}$. 
\end{defin}

From \cite[Theorem 3.3]{GRu}, if  $X$ is a special Perazzo cubic hypersurface, 
 then $Z^*$ is a hypersurface of the linear span   $\langle Z^* \rangle = \P^k \subset \Sing (X) $, that is $k=\mu+1$. 

We distinguish two types of special Perazzo cubic hypersurfaces.

\begin{defin}
A cubic hypersurface with vanishing hessian, not a cone, with $\codim (Z) = 1$ will be called a {\it Minimal Cubic Hypersurface} if $\mu = \dim Z^* = 1.$
\end{defin}
From \cite[Lemma 2.10]{GRu}, if $\mu = \dim Z^* = 1,$ then $X$ is a special Perazzo cubic hypersurface.

\begin{defin}
A cubic hypersurface with vanishing hessian, not a cone, with $\codim (Z) = 1$ will be called a {\it Maximal Cubic Hypersurface} if $$2\dim(Z^*)+2=N,$$ i.e. if $\dim(Z^*)$ is maximal, see (\ref{N>=}).
\end{defin}
From \cite[Theorem 3.1, and Theorem 3.3]{GRu}, Maximal Cubic Hypersurfaces are special Perazzo cubic hypersurfaces.

From \cite[Theorem 4.8]{GRu}, if $X=V(f) \subset \P^{N}$ is a special Perazzo hypersurface, then $Z^*$  is an irreducible component of a determinantal hypersurface $\Delta \subset \P^{\mu +1}$ in its linear span, in particular, $\langle Z^* \rangle \subsetneq \P^{N}$. (recall that $\mu+1<N$).

If $X \subset \P^{N}$, with $N \leq 6$ is a cubic hypersurface not a cone with vanishing hessian, then $X$ is a special cubic Perazzo hypersurface, in particular, $\codim (Z) = 1$ and $\langle Z^* \rangle \subset \Sing (X)$. If $N \geq 7$, there are examples of cubic hypersurfaces with vanishing hessian that are not special Perazzo hypersurfaces and also examples with $\codim(Z)=1$, see \cite[Section 6]{GRu}, but in all of these examples we still have $\langle Z^* \rangle \subset \Sing (X)$, which is equivalent to Perazzo's canonical form.


\begin{rmk}
For $N\leq 6$ we always have $\codim(Z)=1$ (see \cite[Theorem 5.4 and Theorem 5.7]{GRu}). For $N>6$ there are examples of cubics with vanishing Hessian such that $\codim(Z)>1$ (see \cite[Example 3]{GRu}). 
The condition $\langle Z^* \rangle \subset \Sing (X)$ is equivalent to saying that $Z$ is a cone.
\end{rmk}

On the other hand, for any cubic hypersurface $X=V(f)\subset \P^N$, if there exists a linear space $\P^k \subset \Sing (X) $, then from \cite[Proposition 4.1]{GRu}, $f$ is projectively equivalent to: 

\[f=\displaystyle \sum_{i=0}^k x_ig_{i}+h. \]

Here $g_{i} \in \K[x_{N-m+1},\ldots, x_N]_2$ and $h\in \K[x_{k+1},\ldots,x_N]_3$ and $k+m\leq N$.  Notice that $2\leq m \leq k$ implies $\hess_f=0$ since in this case the partial derivatives of $f$ are algebraically dependent. We recall that the algebraic dependence of $g_i$ is equivalent to $\langle Z^* \rangle \subset \Sing (X).$

\subsection{The two extremal families}\label{thetwofamilies}

It is easy to see that the form $f$ is {\it minimal} if $m=k=2$ and that $f$ is {\it maximal} if $N=2k=2m$. These two families correspond to special Perazzo cubic forms.

\begin{enumerate}
    \item {\it The minimal family}

The minimal family of special Perazzo cubics consists of $f\in \K[x_0,\ldots, x_N]$ with $\dim Z^*=1$. In this case $f$ is projectively equivalent to:
\begin{equation}\label{min}
    f=x_0g_{0}+x_1g_{1}+x_2g_{2}+h
\end{equation}
Here $g_i\in \K[x_{N-1},x_N]$ and $h\in \K[x_3,\ldots,x_N]$. Since $X$ is a special Perazzo hypersurface, we have $Z^*\subset \langle Z^* \rangle=\P^2\subset \Sing(X)$.

\item {\it The maximal family}

The maximal family of special Perazzo cubics consists of $X=V(f) \subset \P^{2k}$ with $\dim Z^*=k-1$. Therefore, 
$Z^*\subset \langle Z^*\rangle=\P^k \subset \Sing(X)$. Putting $N=2k$, we have that $f\in \K[x_0,\ldots, x_N]$ is projectively equivalent to
\begin{equation}\label{max}
    f=x_0g_{0}+x_1g_{1}+\ldots+x_kg_{k}+h.
\end{equation}
Here $g_{i},h\in \K[x_{k+1},\ldots,x_N]$.

\begin{rmk}\label{rmkI^2}
We have the following remarks about these two families:

\begin{enumerate}
    \item The closure of the maximal family coincides with the cubics in $I^2$, for some ideal of a $k$-plane $I$. Indeed, let $I=\langle x_{k+1},\ldots,x_N \rangle$ be such an ideal,  the degree three part of $I^2$ is $\textrm{Sym}_{1}(x_0,\dots ,x_k)\otimes \textrm{Sym}_{2}(I)\oplus \textrm{Sym}_{3}(I)$. 
So a cubic $f\in I^2$ is of the form (\ref{max}).
This description will be useful for constructing a parameter space for cubic polynomials within the maximal family; see Section \ref{parameterMax}. 

    \item    There exist cubics whose canonical form is in the closure of maximal family, but in fact belongs to the minimal family, in other  words, there are cubics in the maximal family that specializes to the minimal family.  To illustrate this, consider the cubic
$$f = x_{0}x_{4}^{2} + x_{1}x_{4}x_{5} + x_{2}x_{5}^{2}+x_{3}x_{6}^{2} \ \in \K[x_{0}, \ldots, x_{6}].$$
At a first glance, we see  that $f \in (x_{4}, x_{5},x_{6})^{2}$, hence $\P^{3} = V(x_{4}, x_{5},x_{6}) \subset \Sing (X)$, and therefore $\hess_{f} = 0$. On the other hand, explicit computations show us that $\dim (Z^*) = 1$, and $Z^* \subset \P^{2} = V(x_3, x_4, x_5, x_6) \subset \P^{3}$, so $f$ is in the minimal family. For more details, see \cite{GRu}.

In fact, the cubic $f$ can be obtained as the limit as $t\rightarrow 0$ of the following family (consisting of maximal cubics for $t\neq 0$): $$f_t = x_{0}x_{4}^{2} + x_{1}x_{4}x_{5} + x_{2}(x_{5}^{2}+tx_{6}^{2})+x_{3}x_{6}^{2}.$$
\end{enumerate}
\end{rmk}

\begin{rmk}\label{Ns}
Let $X=V(f) \subset \mathbb{P}^{N}$ be a cubic hypersurface with vanishing hessian, not a cone. By \cite{GRu}, we have the following:
\begin{itemize}
    \item For $N=4$, Theorem 5.1 states that $f$ is maximal. Conversely, by Theorem 5.2, we have  dim$(Z^{*}) = 1$. Therefore, in this case, the minimal and maximal families coincide. 
    
    \item For $N=5$, by Theorem 5.3,  all cubic hypersurfaces with vanishing hessian, not cones, belong to the minimal family.

    \item For $N=6$, we have dim$(Z^{*}) \leq 2$. If dim$(Z^{*}) = 1$, then $f$ is minimal, and if dim$(Z^{*}) = 2$, $f$ is maximal.
 
    \item  For $N\geq 7$, the authors provide examples of cubic hypersurfaces with vanishing hessian, not cones,  which are not special Perazzo hypersurfaces. 
\end{itemize}
\end{rmk}

\end{enumerate}

\section{Artinian Gorenstein algebras and Lefschetz properties}
 Recall from the Introduction that if  $A$ be an Artinian standard graded $\K$-algebra, then $A$ has a decomposition $A = \bigoplus_{i=0}^{d}A_i$, as a sum of finite dimensional $\K$-vector spaces with $A_d \neq 0$. Denoting by $h_{i} = \dim_{\K}A_{i}$, the \textit{Hilbert vector} of $A$ is the vector $\Hilb (A) = (1, h_{1}, \ldots, h_{d})$. The integer $d$ is called \textit{socle degree} of $A$. An Artinian algebra is a Gorenstein algebra if and only if $\textrm{dim}_{\K}A_d = 1$ and the bilinear pairing $$A_i \times A_{d-i} \rightarrow A_d$$ induced by the multiplication map is non-degenerate for $0 \leq i \leq d$. In this case there is  an isomorphism $A_i \simeq \textrm{Hom}_{\K}(A_{d-i},A_d)$ for all $i=0,\ldots,d$. In particular, $\textrm{dim}_{\K}A_i =\textrm{dim}_{\K}A_{d-i}$, so the Hilbert vector of $A$ is symmetric, i.e., $h_{i}=h_{d-i}$ for every $i = 0, \ldots, \lfloor \frac{d}{2} \rfloor$.

\subsection{Lefschetz properties and Macaulay-Matlis duality}
\begin{defin}
Let $A$ be a graded Artinian $\K$-algebra. We say that $A$ has  the \textit{Strong Lefschetz Property} (\textit{SLP}), if there exists an element $L \in A_1$ such that the multiplication map $$\cdot L^k:A_i \rightarrow A_{i+k}$$ has full rank for all integers $0\leq i \leq d-1$ and $1 \leq k \leq d-i$. In this case $L$ is called a Strong Lefschetz Element. 
\end{defin}

Consider the polynomial ring 
$S=\K[x_0,\ldots,x_N]$ as a module over the algebra  $Q = \K[X_0,\ldots,X_N]$ where the action is given by differentiation, i.e.,  $X_i = \frac{\partial}{\partial x_i}$. If $f \in S_{d}$ is a homogeneous polynomial of degree $d \geq 1$, the \textit{annihilator ideal of $f$} is the homogeneous ideal $\Ann_{f} = \{\alpha \in Q; \ \alpha(f) =0\}$.
The ideal $\Ann_{f}$ is also called the Macaulay dual of $f$.

By the theory of inverse systems, we obtain the following characterization of standard graded Artinian Gorenstein $\K$-algebras:
\begin{thm}\label{doubleann}
(Double annihilator theorem of Macaulay). Let $I$ be an ideal of $Q$ such that $Q/I$ is a standard graded Artinian $\K$-algebra of socle degree $d$. Then $Q/I$ is Gorenstein if and only if there exists $f \in S_{d}$ such that $I = \Ann_{f}$.
\end{thm}

\begin{proof}
    See \cite{MW}.
\end{proof}

Let $f \in S_d$ be a homogeneous polynomial. We define
$$A_{f}=\dfrac{Q}{\Ann_{f}}$$ as the standard graded Artinian Gorenstein $\K$-algebra of socle degree $d$ associated to $f$. We assume, without loss of generality, that the codimension of $A_f$ is $N+1$, i.e. that $(\Ann_{f})_{1}=0$. This is equivalent to saying that the partial derivatives of $f$ are linearly independent, which means that $X=V(f)$ is not a cone.
For a more general discussion of Macaulay-Matliy duality see \cite{H-W}.


\begin{defin}For each $k \leq \frac{d}{2}$, let   $\{\alpha_1, \ldots, \alpha_{h_k}\}$ be  an ordered $\K$-basis of $A_k$, the \textit{$k$-th Hessian} of $f$ is the matrix $$\Hess_{f}^{k} = [\alpha_{i}\alpha_{j}(f)]_{1 \leq i,j \leq h_k}.$$ Its determinant is denoted by $\hess_{f}^{k}$. 
\end{defin}
Note that if we  consider the basis $\{X_0,\dots, X_N\}$ of $A_1$, then $\Hess_{f}^{1}$, is the classical Hessian matrix. In this case we simply write $\Hess_{f}$ for the matrix and $\hess_{f}$ for  its determinant.

\begin{rmk}
Although the definition of the $k$-th Hessian depends on the choice of a basis of $A_k$, the vanishing of the $k$-th Hessian is independent of this choice. More precisely, a change of basis multiplies the determinant by a nonzero element of the base field $\K$. 
\end{rmk}

The following theorem yields a connection between Lefschetz properties and higher Hessians:
\begin{thm}\label{wtn}
    \cite{MW} Consider $A_f$, where $f \in S_d$ is a homogeneous polynomial. An element $L=a_{0}X_{0} + \ldots + a_{N}X_{N} \in A_1$ is a strong Lefschetz element of $A_f$ if and only if $\hess_{f}^{k}(a_{0}, \ldots, a_{N}) \neq 0$ for all $0 \leq k \leq \lfloor \frac{d}{2} \rfloor$.

\flushright{$\square$}
\end{thm}

\begin{rmk}\label{ch}
By the theorem above, an Artinian Gorenstein algebra associated to a cubic form has the Strong Lefschetz Property if and only if the determinant of the classical Hessian does not vanish identically.
\end{rmk}

\subsection{The correspondence between cubics with vanishing hessian and AG algebras failing SLP}

Let $f \in \K[x_0,\ldots, x_N]_3$ be a reduced cubic form and let $X=V(f) \subset \P^{N}$ be the corresponding cubic hypersurface. Suppose that $X$ is not a cone, then the Artinian Gorenstein algebra $A=A_f$ has Hilbert vector $\Hilb(A) = (1,n, n,1)$ (recall that  $n=N+1$). Conversely, any standard graded Artinian Gorenstein $\K$-algebra having Hilbert function $\Hilb(A) = (1, n, n, 1)$, by Macaulay-Matlis duality, has a presentation $$A = \K[X_0, \ldots, X_N]/ \Ann_f,$$ for some $f \in \K[x_0. \ldots, x_N]_3$. 

From Therem \ref{wtn}, $A$ fails SLP if and only if $\hess_f=0.$ From \cite{DP} we can restrict ourselves to irreducible cubic forms $f$. Therefore, the bijection given by Macaulay-Matlis duality:
\[
\begin{array}{ccc}
\left \{\parbox{5cm}{
  \centering
  \text{Cubic forms } $f \in \mathbb{K}[x_0, \ldots, x_N]$ \\
  \text{not defining a cone}
} \right \}
& \leftrightarrow &
\left \{\parbox{6cm}{
  \centering
  \text{Artinian Gorenstein} $\mathbb{K}$\text{-algebras with Hilbert vector } $(1, n, n, 1)$ \\
} \right \}
\end{array} 
\]

restricts to

\[
\begin{array}{ccc}
\left \{\parbox{5cm}{
  \centering
  \text{Cubic forms } $f \in \mathbb{K}[x_0, \ldots, x_N]$ \\
  \text{not a cone with } $\hess_f =0$
} \right \}
& \leftrightarrow &
\left \{\parbox{6cm}{
  \centering
  \text{Artinian Gorenstein} $\mathbb{K}$\text{-algebras with Hilbert vector } $(1, n, n, 1)$ \text{failing SLP}
  
} \right \}
\end{array} \]

\subsection{Jordan types}

Here, we compute the possible Jordan types of the Artinian Gorenstein algebras in each family presented in   \ref{thetwofamilies}.  In 
$\P^6$, these Jordan types provide a purely algebraic criterion for determining whether an algebra belongs to one family or the other.

Let us recall some definitions and results about the Jordan type of Artinian algebras. Let $A$ be an Artinian $\K$-algebra.  Given $L \in A_1$,  consider the multiplication map $\cdot L: A \rightarrow A$. Since $A$ is Artinian, the map $\cdot L$ is nilpotent; consequently, it only has zero eigenvalues, the Jordan decomposition of this map induces a partition of $\textrm{dim}_{\K}(A)$ which we denote by  $\mathcal{J}_{A,L}$ and call the \textit{Jordan type} of $A$ with respect to $L$. Without loss of generality, we consider the partition in a non-increasing order. 

If $A_f = Q/\Ann_{f}$, in \cite{CG} is proved that the Jordan type of $A_f$ with respect to  $L$ depends only on the rank of the Hessian of $f$ in $L^{\perp}$.  Here we are using  the following notation: 
 for $L = a_{0}X_{0} + \cdots + a_{N}X_{N} \in A_{1}$,  $L^{\perp}=[a_0:\dots:a_N]$ is  the corresponding point in $\P^N$.


In fact, for $A_f$ of socle degree $3$, in \cite{CG}, the authors proved the following result.


\begin{prop}\label{Jordan}
Let $f \in S_{3}$ be a cubic form, $A_{f}$ its associated Artinian Gorenstein algebra and  $L \in A_{1}$. The Jordan type of $A_{f}$  with respect to $L$ is $$\mathcal{J}_{A_{f},L} = 4^{1} \oplus 2^{r-1} \oplus 1^{2(N+1-r)},$$
 where  $r=\rk(\Hess_{f}(L^{\perp})) $.
\flushright{$\square$}
\end{prop}

\subsubsection{Jordan types for the minimal family}\label{subsecJordanmin}
Let $X=V(f) \subset \P^{N}$ be a general cubic hypersurface in the minimal family. Then, by (\ref{min}), we know that $f$ has normal form $$f=x_0g_{0}+x_1g_{1}+x_2g_{2}+h,$$ where $g_{i}\in \K[x_{N-1},x_N]_{2}$ and $h\in \K[x_3,\ldots,x_N]_{3}$. In this case, we have $\codim(Z) = 1$, $\dim (Z^{*}) = 1$ and being $h$ general,  $\dim (X^{*}) = N-2$.


We analyze the possible Jordan types of $A_{f}$ with respect to $L\in A_{1}$. Using Proposition \ref{Jordan} above, it suffices to determine the rank of  $\Hess_{f}(L^{\perp})$.

\begin{lema}\label{Hessmin}
Notation as above, the Hessian matrix decomposes as $\Hess_{f}=M_0+M_1$, where $M_0$ has coefficients in $\K[x_3,\ldots,x_N]$ and $M_1=\left(\begin{array}{c|cc@{}}
0 & 0 &  \\
\hline
0 & M&
\end{array}\right)$ where $M$ is a $2\times 2$ matrix with coefficients  in $\K[x_0,x_1,x_2]$.   
\end{lema}

\begin{proof}
    The Hessian matrix $\Hess_{f}$ is of the form

$$
\Hess_{f}= \left(\begin{array}{@{}ccc|ccc@{}}
    0 & 0 & 0 &  &  &  \\
    0 & 0 & 0 &    & G &  \\
    0 & 0 & 0 &  &  &  \\\hline
     &  &  &  &  &  \\
     & G^{T} &  &  & H &  \\
          &      &       &       &       & 
  \end{array}\right)
$$
where $G$ is the $3 \times (N-2)$ matrix given by 
$$
G=\left(\begin{array}{@{}ccc|cc@{}}
0  & \ldots  & 0 & \frac{\partial g_0}{\partial x_{N-1}} & \frac{\partial g_0}{\partial x_{N}}\\
0  & \ldots  & 0 &\frac{\partial g_1}{\partial x_{N-1}} & \frac{\partial g_1}{\partial x_{N}}\\
0 & \ldots & 0 & \frac{\partial g_2}{\partial x_{N-1}} & \frac{\partial g_2}{\partial x_{N}}\\
\end{array}\right),
$$
and  $H$ is a $(N-2)\times(N-2)$ matrix of the form 

$$
 H=\left(\begin{array}{@{}ccc|ccc@{}}
&  &  &  &  &  \\
& H_0 &  &  & H_1 &  \\
&  &  &  &  &  \\\hline
&  &  &  &  &  \\
& H_1^{T} &  &  & l_1 + h_{(N-1)(N-1)} & l_{2}+h_{(N-1)N} \\
&      &       &       &   l_{3}+h_{N(N-1)}     & l_{4}+h_{NN}
  \end{array}\right).
$$

Where $h_{ij}:=\frac{\partial h}{\partial x_i\partial x_j}$,
 $H_0$ is a square matrix of order $N-4$ with $(H_0)_{ij}=h_{ij}\in \K[x_3,\ldots,x_N]$ for  $i,j \in \{3, \ldots, N-2\}$ and $H_1$ is a $(N-4)\times 2$ matrix with $(H_1)_{ij}=h_{ij}\in \K[x_3,\ldots,x_N]$ for $i \in \{3, \ldots, N-2\}$, $j\in \{N-1,N\}$, and $l_t$ are linear forms in the variables $x_{0}, x_{1}, x_{2}$, for $t \in \{1,2,3,4\}$.

\end{proof}
Denote by $r(L^{\perp})$ the rank of the Hessian matrix in $L^{\perp}$. We have:
\begin{itemize}
    \item If $L^{\perp} \in X=V(f)$, by Lemma 7.2.8 in \cite{Ru}, we have $\dim X^{*} \leq r(L^{\perp})-2 \leq N-2$. Since $h$ is generic, $\dim X^{*} = N-2$. Therefore, $r(L^{\perp})=N$. This also shows that if $L^{\perp} \in \P^{N}$ is generic, then $r(L^{\perp})=N$;
    
    \item If $L^{\perp} \in \langle Z^{*}\rangle = V(x_{3}, \ldots, x_{N})$, then by Lemma \ref{Hessmin}, the rank of $\Hess_{f}(L^\perp)$ is the rank of the matrix 
$$M:=\left[\begin{array}{cc}
    l_{1} & l_{2} \\
    l_{3} & l_{4}
\end{array}\right] \text{ in } L^\perp.$$
Therefore, if $L^{\perp} \in \langle Z^{*} \rangle$ is generic,  $r(L^{\perp})=2$. 
\item Defining $\Delta$ as the zero locus of $\text{det}M$, if $L^{\perp} \in \Delta$, $r(L^{\perp})=1$. 
\end{itemize}

Therefore
\begin{displaymath}
r(L^\perp) = \left\{ \begin{array}{ll}
N, & \textrm{if $L^{\perp} \in \P^{N} \setminus \langle Z^{*}\rangle$}\\
2, & \textrm{if $L^{\perp} \in \mathcal{U}:= \langle Z^{*}\rangle\setminus \Delta$}\\
1, & \textrm{if $L^{\perp} \in \Delta$.}
\end{array} \right.
\end{displaymath}

\begin{prop}
    Let $A\in \Gor(1,n,n,1)$ be an algebra  in the minimal family. Then the possible Jordan types for $A$ are:
\begin{displaymath}
\mathcal{J}_{A_{f},L} = \left\{ \begin{array}{ll}
4^{1} \oplus 2^{N-1} \oplus 1^{2}, & \textrm{if $L^{\perp} \in \P^{N} \setminus \langle Z^{*}\rangle$}\\
4^{1} \oplus 2^{1} \oplus 1^{2(N-1)}, & \textrm{if $L^{\perp} \in \mathcal{U}:= \langle Z^{*}\rangle\setminus \Delta$}\\
 4^{1} \oplus 1^{2N}, & \textrm{if $L^{\perp} \in \Delta$.}
\end{array} \right.
\end{displaymath}

\end{prop}
\begin{proof}
Use Proposition \ref{Jordan} with the possible values to $r(L^\perp)$ obtained.
    
\end{proof}

\subsubsection{Jordan types for the maximal family}

Let $N=2k$ and $X = V(f) \subset \P^{N}$ be a cubic hypersurface in the maximal family. Then codim$(Z) = 1$, dim$Z^{*}=k-1$ and  $\langle Z^{*} \rangle = \P^{k} = V(x_{k+1}, \ldots, x_{2k})$.
By (\ref{max}), we know that $f$ is projectively equivalent to  $$\sum_{i=0}^{k}x_{i}g_{i} + h$$ where $h, g_{i} \in \K[x_{k+1}, \ldots, x_{2k}]$, $\textrm{deg}(h)=3$ and $\textrm{deg}(g_{i})=2$.

\begin{lema}\label{Hessmax}
For $f$ in the maximal family, the Hessian matrix $\Hess_{f}=M_0+M_1$, where $M_0$ has coefficients in $\K[x_{k+1},\ldots,x_N]$ and $M_1=\left(\begin{array}{c|cc@{}}
0 & 0 &  \\
\hline
0 & M&
\end{array}\right)$ where $M$ is a $k\times k$ matrix with coefficients  in $\K[x_0,\dots ,x_k]$.   
\end{lema}

 

\begin{proof}
    
The Hessian matrix Hess$_{f}$ is given by 

$$
\Hess_{f}= \left(\begin{array}{@{}ccc|ccc@{}}
    0 & \ldots & 0 &  &  &  \\
    \vdots & \ddots & \vdots &  & G &  \\
    0 & \ldots & 0 &  &  &  \\\hline
     &  &  &  &  &  \\
     & G^{T} &  &  & M+H &  \\
          &      &       &       &       & 
  \end{array}\right)
$$

Where $G$ is a $(k+1)\times k$ matrix, such that $(G)_{ij}=\frac{\partial g_i}{\partial x_{j+k}}\in \K[x_{k+1}, \ldots, x_{2k}]$;   
$H=\Hess_{h}$, i.e. $(H)_{ij}\in \K[x_{k+1}, \ldots, x_{2k}]$ and 
     $M$ is a $k\times k$ matrix, such that for $1\leq i,j\leq k$, 
$(M)_{ij}=\sum_{l=0}^{k}x_l \frac{\partial^2 g_l}{\partial x_{i+k}\partial x_{j+k}}\in \K[x_{0}, \ldots, x_{k}]$.

\end{proof}




Notation as before, we have: $r(L^{\perp}) \leq N$. Moreover:
\begin{itemize}
    \item If $L^{\perp} \in X$ is a general point, by Lemma $7.2.8$ in \cite{Ru}, we have $r(L^{\perp}) = N$, then  for generic $L^{\perp}\in \P^N$, $r(L^{\perp})=N$. 
\item If $L^{\perp} \in  \langle Z^{*} \rangle = \P^{k} =V(x_{k+1}, \ldots, x_{2k})$, using Lemma \ref{Hessmax}, we have 
Hess$_{f}(L^{\perp})= \left(\begin{array}{@{}c|c@{}}
    0 & 0  \\
    \hline 
    0 & M_{|\P^k} 
  \end{array}\right)$.
So in this case $r(L^{\perp})\leq k$.

\item  If $\Delta$ denotes  the zero locus of $det(M_{|\P^k})$, we have that for $L^{\perp}\in \Delta$,  $r(L^{\perp})<k$.

\end{itemize}

Therefore
\begin{displaymath}
r(L^{\perp}) = \left\{ \begin{array}{ll}
N, & \textrm{if $L^{\perp} \in \P^{2k}\setminus \langle Z^{*} \rangle$}\\
k, & \textrm{if $L^{\perp} \in \langle Z^{*} \rangle\setminus \Delta$}\\
\leq k-1, & \textrm{if $L^{\perp} \in \Delta$.}
\end{array} \right.
\end{displaymath}

\begin{prop}
    Let $A\in \Gor(1,n,n,1)$ be an algebra  in the maximal family. Then the the possible Jordan types for $A$ are:

\begin{displaymath}
\mathcal{J}_{A_{f},L} = \left\{ \begin{array}{ll}
4^{1} \oplus 2^{N-1} \oplus 1^{2}, & \textrm{if $L^{\perp} \in \P^{N} \setminus \langle Z^{*}\rangle$}\\
4^{1} \oplus 2^{k-1} \oplus 1^{2(k+1)}, & \textrm{if $L^{\perp} \in \langle Z^{*} \rangle\setminus \Delta$}\\
 4^{1} \oplus 2^{n-1} \oplus 1^{2(N+1-n)}, \text{ for some } n<k & \textrm{if $L^{\perp} \in \Delta$.}
\end{array} \right.
\end{displaymath}
\end{prop}

\begin{proof}
Use Proposition \ref{Jordan} with the possible values to $r(L^\perp)$ obtained.
    \end{proof}

Therefore in $\mathbb{P}^{N}$, $N=2k \ \textrm{and} \ k>2$, we can identify the minimal and maximal families in the following way:

\begin{prop}
Let $N=2k$ whit $k>2$, and $A\in \Gor(1,n,n,1)$ failing SLP. If there exists $L \in A_{1}$ whose Jordan type is $\mathcal{J}_{A,L} = 4^{1} \oplus 2^{k-1} \oplus 1^{2(k+1)}$, then $A$ doesn't belong to the minimal family.
In particular, for  $A\in \Gor(1,7,7,1)$, $A$ is in the maximal family if and only if there exists $L \in A_{1}$ whose Jordan type is $\mathcal{J}_{A_{f},L} = 4^{1} \oplus 2^{2} \oplus 1^8$.
\flushright{$\square$}
\end{prop}


\section{Parameter spaces}\label{section5}
Theorem \ref{wtn} states that for  $f\in S_3$,   $A_{f}$  fails to have  SLP if and only if hess$_f = 0$. Let denote by $\mathcal{H}\subset \P(S_3)$ the locus of cubics with vanishing hessian, and by $\mathcal{C}\subset \P(S_3)$ the locus of cubic cones.
In the previous sections, we have shown two families of cubics in  
$\overline{\mathcal{H}\setminus \mathcal{C}}$ that in low dimension ($n=N+1\leq 7$) exhaust the algebras on $\Gor(1,n,n,1)$ failing SLP.

In this section, we construct parameter spaces for the closure of these two families.  Let denote by $\mathbb{X}$ the closure in $\P(S_3)$ of any of the families; the parameter space that we are going to construct describes $\mathbb{X}$ as the birational image of the projectivization of a vector bundle. With this description, we can compute the dimension and degree of $\mathbb{X}$ using techniques of Intersection Theory. 

To compute the degree of $\mathbb{X}$ our principal tools are the Segre and Chern classes of a vector bundle. We refer the reader to \cite[Chapter 3]{FUL}   or \cite[\S10.1 and Ch 5]{EH}  for a systematic treatment of Segre and Chern classes.

 In the following proposition, we give an enumerative interpretation for the degree of a variety $\mathbb{X} \subset \Gor(1,n,n,1)$. 
Recall that any algebra in $\Gor(1,n,n,1)$ is a quotient of  $Q=\C[X_0,\dots,X_N]$.   
 \begin{prop}
 Let $\mathbb{X}\subset \Gor(1,n,n,1)$ be  a  subvariety  of dimension $m$ and degree $d$. Given $L_1,\dots, L_m$ generic linear forms in $Q_1$, the degree $d$ is the number of  algebras in $\mathbb{X}$ that have $L_1,\dots, L_m$ as nilpotents of index $3$.

 \end{prop}
 
 \begin{proof}
 As $\dim \mathbb{X}=m$,  the degree of $\mathbb{X}$ is the number of points in the intersection of $\mathbb{X}$  with a generic codimension $m$ linear space of $\Gor(1,n,n,1)$, and such a linear subspace is the intersection of $m$ generic hyperplanes.
 
 Using the correspondence of $\Gor(1,n,n,1)$ with $\P(S_3)$ of Theorem \ref{doubleann}, we can describe the hyperplanes in $\Gor(1,n,n,1)$. Recall that a hyperplane on $\P(S_3)$ corresponds to a point $P=[a_0:\dots:a_N]\in \P^N$, as follows $$H_P=\{f\in \P(S_3)\,\mid\, f(P)=0\}.$$. On the other hand, recalling the generalized Euler formula:
 $$\text{if } L=a_0X_0+\dots+a_NX_N\in Q_1\, \text{ then }  L^3(f)=3!f(P), $$
we have  that   $f(P)=0$ is equivalent to $L^3\in \Ann_f$.   We conclude that a hyperplane in $\Gor(1,n,n,1)$ is of the form $$H_L=\{A_f\,\mid\, \bar{L}\in A_f \text{ is a nilpotent of index } 3\}.$$ 
 
 So the degree of   $\mathbb{X}\subset \Gor(1,n,n,1)$ has the following interpretation:

 Given $L_1,\dots, L_m$ generic linear forms in $Q_1$, there exist $\deg(\mathbb{X})$ algebras in $\mathbb{X}$ that have $L_1,\dots, L_m$ as nilpotents of index $3$.
    
 \end{proof}


\subsection{Parameter space for the Minimal family} 
In this section, we denote the closure of the minimal family by $\mathbb{X}_{min}$.
A generic element of $\mathbb{X}_{min}$ is projectively equivalent to 
\begin{equation}\label{minp}
    f=x_0g_0+x_1g_1+x_2g_2+h,
\end{equation}
with $g_i\in \C[x_{N-1},x_N]$ and $h\in \C[x_3,\ldots,x_N]$.
So, to parametrize this kind of cubics,  we must choose a $2$-plane $V(I)$ in $\P^N$  and a $(N-2)$-plane $V(J)$  containing it (for example $I=\langle x_3,\ldots,x_N \rangle$ and $J=\langle x_{N-1},x_N \rangle$). Afterward, we have to construct three quadrics in the variables in $J$ (i.e. three elements of $\textrm{Sym}_{2}(J)$) and a cubic in the variables of $I$ (i.e. an element of $\textrm{Sym}_{3}(I)$).

    
    We describe the parameter space in the following theorem.

    \begin{thm}\label{thmMin}
The closure $\mathbb{X}_{min}$ of the minimal family is a rational subvariety of  $\P(S_3)$ of  dimension $5(N-2)+\binom{N}{3}+4$.
    The degree of $\mathbb{X}_{min}$ is given by the top Segre class $s_{m}(\mathcal{E})$ of a vector bundle $\mathcal{E}$ over the flag variety $\mathbb{F}(2,N-2,N+1)$, and can be computed using the Script in \ref{scriptminN}.
\end{thm}
\begin{proof}
    Consider the Grassmannian $\grass(N-2,S_1)$ of $2$-planes in $\P^N$. Denote by  $\mathcal{T}_{1}$ the tautological vector bundle of rank $N-2$, which  fits into the  tautological sequence 
\begin{equation}\label{tau1}
    0 \rightarrow \mathcal{T}_{1} \rightarrow \OX_{\grass(N-2,S_1)} \otimes S_1 \rightarrow \mathcal{Q}_{1} \rightarrow 0
\end{equation} For exemple, the fiber of  $\mathcal{T}_{1}$ over the $2$-plane $V(x_3,\ldots,x_N)$ is the subspace 
$I=[ x_3,\ldots,x_N]_{\C}\subset S_1$. 

Now consider $\grass(2,\mathcal{T}_{1})$, the Grassmannian of rank $2$ subundles of $\mathcal{T}_1$  with structure map $\rho: \grass(2,\mathcal{T}_1) \rightarrow \grass(N-2,S_1)$. For this variety, we have the following tautological sequence 
$$0 \rightarrow \mathcal{T}_{2} \rightarrow \rho^{*}\mathcal{T}_1 \rightarrow \mathcal{Q}_{2} \rightarrow 0$$ 
where $\mathcal{T}_2$ is a vector bundle of rank $2$, whose fiber over $(I,J)\in \grass(2,\mathcal{T}_{1})$ is $J$.

Observe  that $\grass(2,\mathcal{T}_1)$ is in fact the flag variety $\mathbb{F}:=\mathbb{F}(2,N-2,N+1)$. It has dimension $5(N-2)-4$.


Now consider the multiplication map:

$$\varphi: \textrm{Sym}_{2}(\mathcal{T}_{2}) \otimes S_1 \to  S_{3}$$ 
 given by $\varphi\left(\sum_{i}a_{i} \otimes b_{i}\right) = \sum_{i}a_{i}b_{i}$. It defines a map of vector bundles over  $\mathbb{F}$.
 Let $\mathcal{V}_1=im(\varphi)$, it is a subvector bundle of the trivial vector bundle  $\OX_\mathbb{F}\otimes S_{3}$.  We obtain an  exact sequence
\begin{equation}\label{seq1}
0 \rightarrow ker(\varphi) \rightarrow \textrm{Sym}_{2}(\mathcal{T}_{2}) \otimes S_1 \stackrel{\varphi}{\rightarrow} \mathcal{V}_{1}  \rightarrow 0
\end{equation}
where $ker(\varphi)=\wedge^2 \mathcal{T}_{2}\otimes \mathcal{T}_{2}$ (c.f. \cite{FFG}). So we get  an isomorphism $$\dfrac{\textrm{Sym}_{2}(\mathcal{T}_{2}) \otimes S_1}{ker(\varphi)} \stackrel{\overline{\varphi}}{\simeq} \mathcal{V}_{1}.$$ 

Next, we consider the following  map of vector bundles: 
\begin{equation*}
T: \mathcal{V}_{1} \oplus \textrm{Sym}_{3}(\rho^*\mathcal{T}_{1})  \rightarrow S_3 
    \end{equation*}

defined by  $T\left(\overline{\sum_{i}a_{i} \otimes b_{i}},h\right)= \overline{\varphi}\left(\overline{\sum_{i}a_{i} \otimes b_{i}}\right) + h = \sum_{i}a_{i}b_{i} + h$. It is not difficult to see that  $ker(T)=\overline{\varphi}\left(\dfrac{\textrm{Sym}_{2}(\mathcal{T}_{2}) \otimes \rho^*\mathcal{T}_1}{ker(\varphi)}\right)$.

Defining $\mathcal{E}= im(T)$, we obtain that the fibers of $\mathcal{E}$ consists of  the cubics of the normal form (\ref{minp}), and we have an exact sequence of vector bundles over $\mathbb{F}$

\begin{equation}\label{seq2}
0 \rightarrow \dfrac{\textrm{Sym}_{2}(\mathcal{T}_{2}) \otimes \rho^*\mathcal{T}_1}{ker(\varphi)} \rightarrow  \mathcal{V}_{1} \oplus \textrm{Sym}_{3}(\rho^*\mathcal{T}_{1}) \rightarrow \mathcal{E} \rightarrow 0.
\end{equation}

From (\ref{seq1}) and (\ref{seq2}), we obtain that $\rk\mathcal{E}=9+\binom{N}{3}$.



By considering the projectivization $\P(\mathcal{E})$ of the vector bundle $\mathcal{E}$, we conclude that $\mathbb X_{min}$ is the image by the second projection $p_2$:

\begin{displaymath}
\xymatrix{ &
\P(\mathcal{E}) \ar[dl]^{p_{1}} \ar[dr]^{p_{2}} & \\
\mathbb F & &\mathbb{X}_{min} \subset \P(S_3) }
\end{displaymath}

We conclude that $\mathbb X_{min}$ is irreducible.
We claim that $p_{2}$ is generically injective. Indeed, for a generic $f\in \mathbb{X}_{min}$, the singular set of $f$ contains a unique $2$-plane, from which we recover $I$. Consider now the differential of $f$, $df\in H^0(\P^N,\Omega_{\P^N}(3))\subset S_2\otimes S_1$. Projecting $df$ from  $S_2\otimes I$ we get $f_{0}dx_0+f_{1}dx_1+f_{2}dx_2$, and by construction, $f_{0},f_{1},f_{2}\in Sym_2([u,v])$ for some $u,v\in I$, then we recover $J$.

To compute the degree of $\mathbb X_{min} \subset \P(S_3)$, we prove  that, in the present setting, 
$\deg \mathbb{X}_{min}=\int s_{m}(\mathcal{E} )\cap [\mathbb F]$, the $m$-Segre class of $\mathcal{E}$, with  $m=\text{dim} \mathbb{F}$.
A similar equality will be used in the following sections, so we prove it in the more general setting. Indeed, by definition of push forward of cycles, we have 
 $p_{2*}[\P(\mathcal{E} )]=\deg(p_2)[ \mathbb{X}_{min}]$. 
 As we have proven that $\deg(p_2)=1$, putting $\nu=\dim \P(\mathcal{E})=\dim \mathbb{X}_{min}$ and  $H=c_{1}(\OX_{\P(S_3)}(1))$,  we obtain
$$
\deg \mathbb{X}_{min}=\int H^\nu\cap[\mathbb{X}_{min}]=
\int H^\nu\cap{}
p_{2*}[\P(\mathcal{E}) ]=
\int p_2^*
H^\nu\cap[\P(\mathcal{E})]$$ where the last equality was obtained from the projection formula. Now,  
$$ \int p_2^*
H^\nu\cap[\P(\mathcal{E})]=
\int\widetilde H^{\nu}\cap [\P(\mathcal{E}) ]
$$ 
where 
$\widetilde H=c_1(\OX_{\mathcal{E}}(1))
$.

Set $e=\rk(\mathcal{E})$. Thus $\dim \P(\mathcal{E})=e-1+m$.
Hence projection onto the basis $\mathbb{F}$ gives
$$
\int \widetilde H^{\nu}\cap [\P(\mathcal{E})]=\int
p_{1*}(\widetilde H^{\nu}\cap p_{1}^*[\mathbb{F}])=
\int s_{m}(\mathcal{E})\cap[\mathbb{F}].
$$ 
The last equality is the definition of Segre class, c.f. \cite[Chapter 3]{FUL}.

Observe that for the minimal family, we have 
$\dim \mathbb{X}_{min}=\dim \P(\mathcal{E})=e-1+m=4+\binom{N}{3}+5(N-2)$.

To compute $s_{m}(\mathcal{E})$, using sequence (\ref{seq2}) and Whitney formula (c.f. \cite[Thm 3.2]{FUL}),  we have:
\begin{equation}\label{W1}
    s(\mathcal{E})=c(\dfrac{\textrm{Sym}_{2}(\mathcal{T}_{2}) \otimes \rho^*\mathcal{T}_1}{ker(\varphi)})s(\mathcal{V}_{1} \oplus \textrm{Sym}_{3}(\rho^*\mathcal{T}_{1}))
\end{equation}

Other applications of the Whitney formula give us:
$$c(\dfrac{\textrm{Sym}_{2}(\mathcal{T}_{2}) \otimes \rho^*\mathcal{T}_1}{ker(\varphi)})=c(\textrm{Sym}_{2}(\mathcal{T}_{2}) \otimes \rho^*\mathcal{T}_1)s(ker(\varphi))$$
and 
$$s(\mathcal{V}_{1} \oplus \textrm{Sym}_{3}(\rho^*\mathcal{T}_{1}))=s(\mathcal{V}_1)s(\textrm{Sym}_{3}(\rho^*\mathcal{T}_{1}))$$

On the other hand, by sequence (\ref{seq1}) we have 
$$s(\mathcal{V}_{1})=c(ker(\varphi))s(\textrm{Sym}_{2}(\mathcal{T}_{2}) \otimes S_1)$$

Substituting the above  equalities in (\ref{W1}) and using the fact that Segre and Chern classes are inverses to each other we obtain:
\begin{equation}\label{segre}
    s(\mathcal{E})=c(\textrm{Sym}_{2}(\mathcal{T}_{2}) \otimes \rho^*\mathcal{T}_1)s(\textrm{Sym}_{3}(\rho^*\mathcal{T}_{1}))s(\textrm{Sym}_{2}(\mathcal{T}_{2}) \otimes S_1)
\end{equation}

To simplify (\ref{segre})  we twist equation (\ref{tau1})  by $\textrm{Sym}_{2}(\mathcal{T}_{2})$ and use Whitney formula to obtain 
$s(\textrm{Sym}_{2}(\mathcal{T}_{2}) \otimes S_1)c(\textrm{Sym}_{2}(\mathcal{T}_{2}) \otimes \rho^*\mathcal{T}_1)=s(\textrm{Sym}_{2}(\mathcal{T}_{2}) \otimes \rho^*\mathcal{Q}_1)$. Finally, we obtain 
$$ s(\mathcal{E})=s(\textrm{Sym}_{3}(\rho^*\mathcal{T}_{1}))s(\textrm{Sym}_{2}(\mathcal{T}_{2}) \otimes \rho^*\mathcal{Q}_1)=s(\textrm{Sym}_{3}(\rho^*\mathcal{T}_{1})\oplus \textrm{Sym}_{2}(\mathcal{T}_{2}) \otimes \rho^*\mathcal{Q}_1).
$$

\end{proof}
Next, we parametrize and compute the dimension and degree of the intersection  $\mathbb{X}_{min}\cap \mathcal{C}$.
There are three types of cones in $\mathbb{X}_{min}$, classified according to the position of the vertex. We denote these cones by $\mathcal{C}_i$; for  $i=1,2,3$. Let the vertex be $p=V(I_0)$. Then we have the following inequalities: $N-3\leq \dim(I\cap I_0)\leq N-2$ and $1\leq \dim(J\cap I_0)\leq 2$.
These inequalities will give us the following cases: 
\begin{enumerate}
    \item If $\dim(I\cap I_0)=N-3$ and $\dim(J\cap I_0)=1$, we can suppose that $I_0=\langle x_3,\dots ,x_{N-1}\rangle $. In this case $f \in \mathbb{X}\cap \mathcal{C}_1$ is projectively equivalent to $x_0x_{N-1}^2+h(x_3,\dots ,x_{N-1})$.  
\item If $\dim(I\cap I_0)=N-3$ and $J\subset  I_0$, we can suppose that $I_0=\langle x_4,\dots ,x_{N-1},x_N\rangle $, and $f \in \mathbb{X}_{min}\cap \mathcal{C}_2$ is projectively equivalent to
$x_0x_{N-1}^2+x_1x_{N-1}x_N+x_2x_{N}^2+h(x_4,\dots ,x_{N-1},x_N)$. It is easy to see that 
$\mathbb{X}_{min}\cap \mathcal{C}_1\subset \mathbb{X}_{min}\cap \mathcal{C}_2$.
\item If $J\subset I\subset I_0$, then the vertex $p=V(I_0)\subset V(I)=\P^2$.
In this case we can suppose that $I_0=\langle x_1,x_2,\dots ,x_{N-1}, x_N\rangle $. Thus  $f \in \mathbb{X}\cap \mathcal{C}_3$ is projectively equivalent to $x_{1}x_{N-1}x_{N} + x_{2}x_{N}^2 + h(x_{3},\dots ,x_{N})$.

We have $\mathbb{X}_{min}\cap \mathcal{C}_1\subset \mathbb{X}_{min}\cap \mathcal{C}_2\subset \mathbb{X}_{min}\cap \mathcal{C}_3$.
\end{enumerate}

To parametrize $\mathbb{X}_{min}\cap \mathcal{C}_3$, we proceed as follows: Start with the Grassmannian $\grass(N,S_1)$, let  $\T_0$ denote its tautological bundle. Consider  the Grassmannian $\grass(N-2,\T_0)$, parametrizing rank $(N-2)$ subundles of $\T_0$, with  natural structure map $\rho_0:\grass(N-2,\T_0)\to \grass(N,S_1)$. Let  $\T_1$ be the tautological bundle on $\grass(N-2,\T_0)$. Next, consider the Grassmannian $\grass(2,\T_1)$, which parametrizes rank $2$ subundles of $\T_1$,  with structure map $\rho_1:\grass(2,\T_1)\to \grass(N-2,\T_0)$. Denote by $\T_2$ the tautological bundle. 
This construction yields a  tower of fibrations: 

\[
\grass(2, \mathcal{T}_1) 
\xrightarrow{\ \rho_1\ }
\grass(N-2, \mathcal{T}_0) 
\xrightarrow{\ \rho_0\ }
\grass(N, S_1)
\]

We obtain that $\grass(2,\T_1)$ is the flag variey $\mathbb{F}:=\mathbb{F}(2,N-2,N,N+1)$ and $\dim \mathbb{F}=N+2(N-2)+2(N-4)=5(N-2)-2$.

The construction is completely analogous to what we did in the proof of Theorem \ref{thmMin}. Consider
 $$\wedge^2\T_2\otimes \T_2\to \textrm{Sym}_{2}(\T_2)\otimes \T_0\to \mathcal{V}_1$$

 and $$\frac{\textrm{Sym}_{2}(\T_2)\otimes \T_1}{\wedge^2\T_2\otimes \T_2}\to \mathcal{V}_1\oplus \textrm{Sym}_{3}(\T_1)\to \mathcal{F}.$$
Where to simplify the notation, we are omitting the pullbacks by $\rho_0$ and $\rho_1$.
We obtain a fiber bundle $\mathcal{F}$
such  that $\rk(\mathcal{F})=3N-2+\binom{N}{3}-(3(N-2)-2)=\binom{N}{3}+6$, and $\mathbb{X}_{min}\cap \mathcal{C}_3$ is the projection on the second factor of $\P(\mathcal{F})$. This projection is generically injective; therefore $$\dim (\mathbb{X}_{min}\cap \mathcal{C}_3)=5(N-2)-2+\binom{N}{3}+5=5(N-2)+\binom{N}{3}+3.$$
 Observe that $\mathbb{X}_{min}\cap \mathcal{C}_3$ is a divisor in $\mathbb{X}$.
From the construction of $\mathcal{F}$ we can obtain the degree of this divisor. With the notation as above,  we get the following result.

\begin{prop}\label{minwithcones}
The variety  $\mathbb{X}_{min}\cap \mathcal{C}_3$ is a divisor in $\mathbb{X}_{min}$ of degree given by the top Segre class $s_{m}(\textrm{Sym}_{3}(\mathcal{T}_{1})\oplus \textrm{Sym}_{2}(\mathcal{T}_{2}) \otimes \mathcal{Q}_1)$, where $m=5(N-2)-2$, and can be computed using the Script in \ref{script1conesminimal}.
    
\end{prop}

\subsection{Parameter space of the Maximal family}\label{parameterMax}
Next, we consider $N=2k$ and denote by $\mathbb{X}_{max}$ the closure of the  maximal family. We observe that, by Remark \ref{rmkI^2},  $\mathbb{X}_{max}$ coincides with the locus in $\P(S_3)$ of cubics in $I^2$ for some ideal of a $k$-plane $I$. A generic element of $\mathbb{X}_{max}$ is projectively equivalent to

\begin{equation}\label{p6Mgeral}
 f:=\sum_{i=0}^{k}x_{i}g_{i}(x_{k+1},\dots ,x_{N}) + h(x_{k+1},\dots ,x_{N})  
\end{equation}
where $g_{0}, g_{1}, \ldots, g_{k}$ are quadratic forms  and $h$ is a cubic form in the variables $x_{k+1},\dots ,x_{N}$. 


\begin{thm}\label{thmMax}
For $N=2k$, the maximal family of cubics in $\P^N$ is a rational, projective, irreducible variety of dimension $$\dim \mathbb{X}_{max}=(k+1)(\binom{k+1}{2}+k)+\binom{k+2}{3}-1=\frac{1}{6}(4k^3+15k^2+11k-6)$$ and whose  degree is given by $s_m(\mathcal{E})$, where
$\mathcal{E}$ is a vector bundle over a variety $\grass(k,S_1)$ of dimension
$m=(k+1)k$.
The degree of this family can be computed using the Script in \ref{script4}.
\end{thm}
\begin{proof}
     Recall the description of cubics in $\mathbb{X}_{max}$ as cubics contained in $I^2$, where $I$ is the ideal of a $k$-plane, given in Remark \ref{rmkI^2}.  
     Consider $\grass(k,S_1)$, the Grassmannian of $k$-planes in $\P^N$ with tautological sequence 
$$0 \rightarrow \mathcal{T} \rightarrow \grass(k,S_{1}) \times S_1 \rightarrow \mathcal{Q} \rightarrow 0.$$ 
Consider  the multiplication map 
$$\varphi:\textrm{Sym}_2(\mathcal{T})\otimes S_1\to  \mathcal O_{\grass(k,S_1)}\otimes S_3.$$
It defines a map of vector bundles whose image parametrizes the set of pairs $(I,f)\in \grass(k,S_1)\times S_3$ such that $f\in I^2$.

 We obtain the following  exact sequence  of vector bundles over $\grass(k,S_1)$:
\begin{equation}\label{varphi}
0 \rightarrow \wedge^3 \mathcal{T}\rightarrow \wedge^2\mathcal{T}\otimes \mathcal{T}\rightarrow \textrm{Sym}_2 \mathcal{T}\otimes S_1\xrightarrow[]{\varphi} \Ee\rightarrow 0
\end{equation} 
where $\Ee={\rm Im}\varphi$.

 Following the above construction, it is not difficult to see that $\rk \mathcal{E}=
 (2k+1)\binom{k+1}{2}+\binom{k}{3}-k\binom{k}{2}
 $.

We have the following projections

\begin{displaymath}
\xymatrix{ &\P(\mathcal{E}) \ar[dl]_{p_{1}} \ar[dr]^{p_{2}} & \\
\grass(k,S_1) & & \mathbb{X}_{max} \subset \P(S_3) }
\end{displaymath}

We claim that the map $p_{2}$ is generically injective. Consequently, the vairty  $\mathbb{X}_{max}=p_2(\P(\mathcal{E}))$ has dimension $\rk \mathcal{E}+\dim \grass(k,S_1)-1=(2k+1)\binom{k+1}{2}+\binom{k}{3}-k\binom{k}{2}+k(k+1)-1=\frac{1}{6}(4k^3+15k^2+11k-6)$.

The proof of the claim follows the same arguments as in the proof of  Theorem \ref{thmMin}: a generic cubic $f\in \mathbb{X}_{max}$ has a unique $3$-plane in its singular set, wich allows us to recover the ideal $I$. 


To compute the degree of $\mathbb{X}_{max}$,  we proceed exactly as in the proof of Theorem \ref{thmMin}. Using the sequence (\ref{varphi}), we have

$$s(\mathcal{E})=s(\textrm{Sym}_2 \mathcal{T}\otimes S_1)c(\wedge^2\mathcal{T}\otimes \mathcal{T})s(\wedge^3 \mathcal{T}).$$


\end{proof}
In what follows, we construct a parameter space for the intersection of $\mathbb{X}_{max}$ with cubic cones in $\P^N$.

Let the vertex of the cone be denoted by $p=V(I_0)$. Then, we have $k-1\leq \dim I\cap I_0\leq k$. Consequently, there are two types of cones: 
\begin{enumerate}
    \item  The case where $\dim I\cap I_0=k-1$, i.e. $p\not\in V(I)=\P^k$. We can suppose that $I=\langle x_{k+1},\dots ,x_{N}\rangle$, and $I_0=\langle x_0,\dots, x_{N-1}\rangle$, 
thus $f$ is projectively equivalent to $\sum_{i=0}^{k}x_{i}g^{i}(x_{k+1},\dots ,x_{N-1}) + h(x_{k+1},\dots ,x_{N-1})$.
We write $\mathbb{X}\cap \mathcal{C}_1$ the intersection with these cones.
    
    \item  The case where $\dim I\cap I_0=k$, i.e $I\subset I_0$ and $p=V(I_0)\in V(I)=\P^k$. In this case we can assume that $I_0=\langle x_1,\dots, x_N\rangle $ and $f$ is projectively equivalent to 
    $\sum_{i=1}^{k}x_{i}g^{i}(x_{k+1},\dots ,x_{N}) + h(x_{k+1},\dots ,x_{N})$.
We write $\mathbb{X}\cap \mathcal{C}_2$ the intersection with these cones.
Observe that $\mathbb{X}\cap \mathcal{C}_1\subset \mathbb{X}\cap \mathcal{C}_2$.
\end{enumerate}

 Next, we parametrize $\mathbb{X}\cap \mathcal{C}_2$. 
Consider the Grassmannian $\grass(N,S_1)$ with tautological bundle $\T_0$. Then consider the Grassmannian $\grass(k,\T_0)$, which parametrizing rank $k$ subundles of $\T_0$, equipped with its  tautological bundle $\T_1$.

This yields  a fibration: 
$$ \rho: \grass(k,\T_0)\to \grass(N,S_1).$$ 

Over $\grass(k,\T_0)$, we have the following  exact sequence  of vector bundles:
\begin{equation}\label{varphi0}
0 \rightarrow \wedge^3 \mathcal{T}_1\rightarrow \wedge^2\mathcal{T}_1\otimes \mathcal{T}_1\rightarrow \textrm{Sym}_2 \mathcal{T}_1\otimes \rho^*\T_0\xrightarrow[]{\varphi} \mathcal{F}\rightarrow 0
\end{equation} 
where $\mathcal{F}={\rm Im}\varphi$.  From this sequence, it follows that $\rk(\mathcal{F})=2k\binom{k+1}{2}+\binom{k}{3}-k\binom{k}{2}$.



We conclude that $\mathbb{X}\cap \mathcal{C}_2$ is the projection onto the second factor of $\P(\mathcal{F})$, and this projection is generically injective. Therefore, $$\dim (\mathbb{X}\cap \mathcal{C}_2)=2k\binom{k+1}{2}+\binom{k}{3}-k\binom{k}{2}+k^2+2k-1.$$
From the construction of $\mathcal{F}$, if follows that the degree of $\mathbb{X}\cap \mathcal{C}_2$ is given by the Segre class $s_{m}(\mathbb{F})$.

Hence, we have the following proposition:
\begin{prop}\label{maximalwithcones}
The variety  $\mathbb{X}_{max}\cap \mathcal{C}_2\subset \mathbb{X}$ has codimension $\binom{k}{2}$, and its degree 
can be computed using the Script in \ref{script1conesmaximal}.
    
\end{prop}

\begin{rmk}
    For the case $N=2k$, the difference in dimensions is $$\dim(\mathbb{X}_{min})-\dim(\mathbb{X}_{max})=\frac{1}{6}(k-2)(4k^2-19k+15).$$ Thus,  for $N=4$, the dimension coincide; for $N=6$, we have  $$\dim(\mathbb{X}_{min})<\dim(\mathbb{X}_{max}),$$ and for $N\geq 8$, $$\dim(\mathbb{X}_{min})>\dim(\mathbb{X}_{max}).$$
\end{rmk}

\subsection{Tangent spaces to \texorpdfstring{$\mathcal{H}$}{H} in members of each family.}

In this section, we compute the dimension of the tangent space to $\mathcal{H}$
at points of the two families under study: $\mathbb{X}_{min}$ and $\mathbb{X}_{max}$. We conclude that for $N\geq 5$, $\mathbb{X}_{min}$ is the reduced part of a non-reduced component of $\mathcal{H}$, while $\mathbb{X}_{max}$ is an irreducible component of $\mathcal{H}$. 
Furthermore, in $\P^6$, we show  that neither family is contained in the closure of the other, see Proposition \ref{relation}.

Consider the morphism  $F=\text{det}\circ Hess$:

\begin{equation}
    S_3\stackrel{Hess}{\longrightarrow} \textrm{M}_{N+1} (S_1)\stackrel{\text{det}}{\longrightarrow} S_{N+1}
\end{equation}

where $\textrm{M}_{m} (S_1)$ denotes the space of $m\times m$ matrices with entries  in $S_1$. 
The $\binom{2N+1}{N}$  coordinates of $F$ define the subscheme of $\P(S_3)$ that we have denoted by $\mathcal{H}$.

Let $f\in \mathcal{H}$, then $T_f\mathcal{H}$ is given by  $\P(\ker(d_f F))$ and  $$d_f F=d_{Hess(f)}\text{det}\circ d_f Hess=d_{Hess(f)}\text{det}\circ Hess.$$ On the other hand, it is  well known  that if $A$ is a square matrix, then $\frac{\partial \text{det}}{\partial a_{ij}}(A)=A_{ij}$ , where $A_{ij}$ is the $(i,j)$-minor of $A$.
Therefore $d_f F: S_3\to S_{N+1}$ is then given by 
$$v\mapsto \sum_{i,j} \delta_{ij} v_{ij},$$ where 
$\delta_{ij}$ is the $(i,j)$-minor of $Hess_f$, and $v_{ij}=\frac{\partial^2 v}{\partial x_i \partial x_j}$.

\subsubsection{Tangent to \texorpdfstring{$\mathcal{H}$}{H} in a point of the minimal family}\label{tangentmin}

A cubic  $f \in \mathbb{X}_{min}$ is projectively equivalent to  
\begin{equation}
f=x_{0}g_{0} + x_{1}g_{1} + x_{2}g_{2} + h(x_{3},\dots ,x_{N-1},x_{N})
\end{equation} 

where $g_{i}$ are quadrics depending on $(x_{N-1},x_N)$ and $h(x_3,\dots ,x_{N-1},x_N)$ is a cubic. 

By Lemma \ref{Hessmin}, the Hessian matrix $\Hess_f$ is of the form:

$$
\Hess_f= \left(\begin{array}{@{}ccc|ccc|ccc@{}}
    0 & 0 & 0 & 0 & \cdots &0 && &  \\
    0 & 0 & 0 & 0 &\cdots &0 && G &  \\
    0 & 0 & 0 & 0 &\cdots &0 &&  &  \\
\hline
     0& 0 & 0 & h_{33} & \cdots &h_{3N-2} & * & * \\
     \vdots& \vdots & \vdots &  & \ddots & & \vdots & \vdots\\
     0& 0 & 0 & h_{3N-2} &\cdots &h_{N-2N-2}  &  *&*\\
     \hline
          &      &   &  *  &  *     &  *  & *& *\\
           &   G^t   &   &  *  &  *     &  *  & *& * \end{array}\right)
$$

Where $G\in M_{3,2}(\textrm{Sym}_{1}(x_{N-1},x_N))$ is the Jacobian of the map $(g_0,g_1,g_2)$.

For $i=0,1,2$, denote by $\Delta_i$ the determinant of the matrix obtained  by deleting the $i$-th row of $G$.
Also, denote by $\Delta$ the determinant of the matrix $\left(\begin{array}{@{}ccc@{}}
    h_{33} & \cdots & h_{3N-2} \\
    & \ddots &\\
        h_{3N-2} & \cdots &h_{N-2N-2} 
  \end{array}\right)$.

It is not difficult to see that:
\begin{enumerate}
    \item If $i\leq j$ and $i,j\in \{0,1,2\}$,  $\delta_{ij}:= (\Hess_f)_{ij}$ is given by $$\delta_{ij}=\Delta_{i}\Delta_j\Delta \in \textrm{Sym}_{4}(x_{N-1},x_N)\Delta. $$ 
\item If $2<i\leq j\leq N$,  $\delta_{ij}=0$.
\end{enumerate}


Thus, the differential   $d_f F: S_3\to S_{N+1}$ is then given by 
$d_fF(v)= \sum_{i,j\in \{0,1,2\}} \delta_{ij} v_{ij}$.

In order to compute the rank of $d_f F$, we decompose $S_3$ as the direct sum of the  following subspaces: 
\[E_1:=\textrm{Sym}_{3}(x_3,\dots ,x_N)\oplus \textrm{Sym}_{2}(x_3,\dots ,x_N)\otimes \textrm{Sym}_{1}(x_0,x_1,x_2),\]
\[E_2:=\textrm{Sym}_{1}(x_3,\dots ,x_{N-2})\otimes \textrm{Sym}_{2}(x_0,x_1,x_2).\]
\[E_3:= \textrm{Sym}_{1}(x_{N-1},x_N)\otimes \textrm{Sym}_{2}(x_0,x_1,x_2) .\]
\[E_4= \textrm{Sym}_{3}(x_0,x_1,x_2)\]
 
 Observe that $d_f F(E_1)=\{0\}$. On the other hand, using that $\delta_{ij}\in \textrm{Sym}_{4}(x_{N-1},x_N)\Delta$, we obtain: 
\begin{enumerate}
    \item 
$d_f F(E_2)\subset \textrm{Sym}_{1}(x_3,\dots ,x_{N-2})\otimes \textrm{Sym}_{4}(x_{N-1},x_N)\Delta$, 
so $$\text{dim}( d_f F(E_2))\leq 5(N-4).$$

\item $d_f F(E_3)\subset \textrm{Sym}_{5}(x_{N-1},x_N)\Delta$, 
so $$\text{dim}(d_f F(E_3))\leq 6.$$
    \item $\text{dim}(d_f F(\textrm{Sym}_{3}(x_0,x_1,x_2)))\leq 10$ .

\end{enumerate}    

Summarizing,  the rank of $d_f F$ is $\leq 5N-4$, and $ \text{dim}\ker (d_f F)\geq \binom{N+3}{3}-(5N-4)$.
So $\text{dim}(T_f\mathcal{H} )\geq \binom{N+3}{3}-(5N-4)-1>\text{dim}\mathbb{X}_{min}$ for all $f\in \mathbb{X}_{min}$. 


\subsubsection{Tangent to \texorpdfstring{$\mathcal{H}$}{H} in a point of the maximal family}

We begin by addressing the case $N=6$. 
We claim that the  cubic 
$$f=\frac{1}{2}(x_0x_4^2+x_1x_5^2+x_2x_6^2+x_3L^2),$$ where $L=x_4+x_5+x_6$,  is a smooth point of $\mathbb{X}_{max}$.
Indeed, we will prove that 
 $\rk(d_{f}F)=38$, so  $\dim T_f\mathcal{H}=84-38-1=45=\text{dim}(\mathbb{X}_{max})$. 



This shows that the open set of regular points in $\mathbb{X}_{max}$ is non-empty, so $\mathbb{X}_{max}$ is a reduced irreducible component of $\mathcal{H}$.

    Let prove the claim:
the Hessian of $f$ is of the form

$$
\Hess_{f}= \left(\begin{array}{@{}cccc|ccc@{}}
    
    0 & 0&0& 0 & x_4 & 0 & 0 \\
    0 & 0&0& 0 & 0 & x_5 & 0 \\
    0 & 0&0& 0 & 0 & 0 & x_6 \\
    0 & 0&0& 0 & L & L & L \\\hline
    x_4 & 0 & 0 & L &  &  \\
    0 & x_5 & 0 & L &  & H \\
    0&  0   & x_6 & L  &       & 
  \end{array}\right)
$$

Define  $G= \left(\begin{array}{@{}ccc@{}}
    
     x_4 & 0 & 0 \\
     0 & x_5 & 0 \\
     0 & 0 & x_6 \\
     L & L & L  
  \end{array}\right)
$, and 
for $i=0,1,2,3$, denote by $\Delta_i$ the determinant of the matrix obtained by deleting the $i$-th row of $G$.

It is not difficult to see that:
\begin{enumerate}
    \item If $i\leq j$ and $i,j\in \{0,1,2,3\}$,  $\delta_{ij}:= (\Hess_f)_{ij}$ is given by $$\delta_{ij}=\Delta_{i}\Delta_j\in \textrm{Sym}_{6}(x_4,x_5,x_6), $$ 
\item If $3<i\leq j\leq 6$,  $\delta_{ij}=0$.
\end{enumerate}

\begin{lema}\label{lema3}
The following sets are linearly independent:
\begin{enumerate}
    \item $\mathcal{A}(3):=\{\delta_{ij}\mid 0\leq i\leq j\leq 3\}$.
\item $\mathcal{B}(3):=$
$$\{x_4\delta_{ij}\mid i,j\neq 0\}\cup \{x_5\delta_{ij}\mid i,j\neq 1\}\cup \{x_6\delta_{ij}\mid i,j\neq 2\}.$$
\end{enumerate}

\end{lema}
\begin{proof}
    Denote by $\Pi=x_4x_5x_6$; then 
$\Delta_0=L\frac{\Pi}{x_{4}}, \Delta_1=L\frac{\Pi}{x_{5}}, \Delta_2=L\frac{\Pi}{x_{6}}$, and $\Delta_3=\Pi$.
To prove (1), suppose that we have $\sum_{0\leq i\leq j\leq 3}\alpha_{ij}\delta_{ij}=0$, or  
\[\alpha_{00}\delta_{00}+(\alpha_{01}\delta_{01}+\dots +\alpha_{03}\delta_{03})+\sum_{1\leq i\leq j\leq 3}\alpha_{ij}\delta_{ij}=0\]
Observe that $\delta_{00}\notin \langle x_4\rangle$, $\delta_{0j}\in \langle x_4\rangle$ for all $j\geq 1$ and $\delta_{ij}\in \langle x_4^2\rangle$ for all $i\geq 1$. So  we can write the above equation in the following way:
\begin{equation}\label{equationx4}
    \alpha_{00}\delta_{00}+x_4A_1+x_4^2A_2=0
\end{equation}
where $$A_1=\Delta_0(\alpha_{01}x_6L+\alpha_{02}x_5L+\alpha_{03}x_6x_5)\equiv$$$$\Delta_0(\alpha_{01}x_6(x_5+x_6)+\alpha_{02}x_5(x_5+x_6)+\alpha_{03}x_6x_5)\text{ mod }\langle x_4\rangle.$$ 
Thus (\ref{equationx4}) gives us $\alpha_{00}=0$, and $A_1\equiv 0$  mod $\langle x_4\rangle$.  From this we obtain $\alpha_{01}=\alpha_{02}=\alpha_{03}=0$. 
Hence \begin{equation}\label{equationx5}
    A_2=\alpha_{11}x_6^2L^2+x_5x_6L(\alpha_{12}L+\alpha_{13}x_6)+x_5^2(\alpha_{22}L^2+\alpha_{23}x_6L+\alpha_{33}x_6^2)=0.
\end{equation}
From (\ref{equationx5}), arguing modulo $\langle x_5\rangle $,  we obtain $\alpha_{i,j}=0$ for all $j\geq i\geq 1$.

To prove (2) observe that for $i,j\neq 0$ we have $x_4\delta_{ij}\in x_4\langle \Delta_1,\Delta_2,\Delta_3\rangle^2=x_4^3\langle x_6L,x_5L,x_5x_6 \rangle^2$. 
Analogously, for $i,j\neq 1$ we have \[x_5\delta_{ij}\in x_5^3\langle x_6L,x_4L,x_4x_6 \rangle^2\]
and for $i,j\neq 2$ we have \[x_6\delta_{ij}\in x_6^3\langle x_5L,x_4L,x_4x_5 \rangle^2.\]

 If  there was a zero linear combination of elements of $\mathcal{B}(3)$, we can write it in the following way:
 \begin{equation}\label{equationx6}
     x_4^3A_1+x_5^3A_2+x_6^3A_3=0
 \end{equation}

then  $A_1\in \langle x_5,x_6\rangle^3$. 
Using that  $A_1\in \langle x_6L,x_5L,x_5x_6 \rangle^2$ we conclude that  $A_1=\lambda_1 (x_5x_6)^2$, with $\lambda_1\in \C$.
Analogously, we find that  
$A_2=\lambda_2(x_4x_6)^2$ and $A_3=\lambda_3(x_4x_5)^2$, with $\lambda_2, \lambda_3\in \C$.
Using equation (\ref{equationx6}) we conclude that $\lambda_1=\lambda_2=\lambda_3=0$.

On the other hand, by definition, 
$x_4^2A_1$ is a linear combination of $\delta_{ij}$, which we have proven to be  linearly independent.
The same holds for $x_4^2A_2$ and $x_6^2A_3$.


\end{proof}

In order to obtain the rank of $d_fF$,  we decompose $S_3$ as a direct sum of the  following subspaces: 
\[E_1:=\textrm{Sym}_{3}(x_4,x_5,x_6)\oplus \textrm{Sym}_{2}(x_4,x_5,x_6)\otimes \textrm{Sym}_{1}(x_0,x_1,x_2,x_3),\]
\[E_2:=\textrm{Sym}_{1}(x_4,x_5,x_6)\otimes \textrm{Sym}_{2}(x_0,x_1,x_2,x_3).\]
\[E_3:= \textrm{Sym}_{3}(x_0,x_1,x_2,x_3) .\]

We have $d_f F(E_1)=\{0\}$. On the other hand, it is possible to prove that:
\begin{enumerate}
    \item $d_f F(E_2)$ contains  $\mathcal{B}(3)$.

\item  $\dim(d_f F(E_3))=20$. In fact,  $d_f F(E_3)$ contains the following set: 
$$\mathcal{C}(3):=\{x_0\delta_{ij}+l_{0,i,j}(x_1,\dots ,x_6) \mid i,j\in \{0,1,2,3\}\}\cup $$$$\{x_1\delta_{ij}+l_{1,i,j}(x_2,\dots ,x_6)\mid i,j\in \{1,2,3\}\}\cup $$$$\{x_2\delta_{ij}+l_{2,i,j}(x_3, \dots, x_6)\mid i,j\in \{2,3\}\}\cup \{x_3\delta_{33}\}.$$  

Using that $\mathcal{A}(3)$ is linearly independent, it is easy to see that $\mathcal{C}(3)$ is linearly independent. 
\end{enumerate}    

Observe that  $\#\mathcal{C}(3)=20$,  $\#\mathcal{B}(3)=18$ and that $\mathcal{C}(3)\cup \mathcal{B}(3)$ is also linearly independent.

Summarizing, as $d_f F(S_3)$ contains $\mathcal{B}(3)\cup \mathcal{C}(3)$, we conclude  that $\rk (d_f F)=38$, so  
 $\text{dim}(T_f\mathcal{H} )=\binom{6+3}{3}-38-1=45=\text{dim}\mathbb{X}_{max}$.

In what follows, we generalize the above computations to  $\P^{2k}$.

\begin{prop}

The following cubic is a non-singular point of $\mathbb{X}_{max}$ in $\P^{2k}$
\[f=x_0x_{k+1}^2+x_1x_{k+2}^2+\dots x_{k-1}x_{2k}^2+x_kL^2,\] where $L=x_{k+1}+\dots +x_{2k}$.
\end{prop}
\begin{proof}
    
The proof is based on the fact  that the sets 
 $\mathcal{A}(k)$ and $\mathcal{B}(k)$ are linearly independent (c.f. Lemma \ref{lemak} below). Indeed,  decompose $S_3$ as a direct sum of the  following subspaces: 
\[E_1:=\textrm{Sym}_{3}(x_{k+1},\dots ,x_N)\oplus \textrm{Sym}_{2}(x_{k+1},\dots ,x_N)\otimes \textrm{Sym}_{1}(x_0,\dots ,x_k),\]
\[E_2:=\textrm{Sym}_{1}(x_{k+1},\dots ,x_N)\otimes \textrm{Sym}_{2}(x_0,\dots ,x_k).\]
\[E_3:= \textrm{Sym}_{3}(x_0,\dots ,x_k).\]
We have that:
 
\begin{enumerate}
    \item $d_f F(E_2)$ contains the following linearly independent set
    $$\mathcal{B}(k):=\{x_{k+1}\delta_{ij}\mid i,j\neq 0\}\cup \{x_{k+2}\delta_{ij}\mid i,j\neq 1\}\cup \dots $$$$\cup \{x_{2k}\delta_{ij}\mid i,j\neq k-1 \}.$$ Observe that $\#\mathcal{B}(k)=k\binom{k+1}{2}$.

\item  $\dim(d_f F(E_3))=\binom{k+3}{3}$. In fact, using that $\mathcal{A}(k)$ is linearly independent, it is easy to see that $d_f F(E_3)$ contains the following linearly independent set: 
\[\mathcal{C}(k):=\{x_0\delta_{ij}+l_{0,i,j}(x_1,\dots ,x_{2k}) \mid i,j\in \{0,\dots , k\}\}\cup \]\[\{x_1\delta_{ij}+l_{1,i,j}(x_2,\dots ,x_{2k})\mid i,j\in \{1,\dots, k\}\}\cup \dots \]
\[\cup \{x_{k-1}\delta_{ij}+l_{k-1,i,j}(x_k, \dots, x_{2k})\mid i,j\in \{k-1,k\}\}\cup \{x_k\delta_{kk}\}.\]  
We remark that $\#\mathcal{C}(k)=\binom{k+3}{3}$.
\end{enumerate}

Therefore we have $\rk d_fF\geq k\binom{k+1}{2}+\binom{k+3}{3}$ and $\dim T_f \mathcal{H}\leq \binom{2k+3}{3}-(k\binom{k+1}{2}+\binom{k+3}{3})-1=\dim \mathbb{X}_{max}$. So $\dim T_f \mathcal{H}=\dim \mathbb{X}_{max}$.

\end{proof}


\begin{lema}\label{lemak}
For $n\geq 3$, consider $\Pi=x_{n+1}\dots x_{2n}$. For $i=0,\dots, n-1$, define 
$\Delta_i=L\frac{\Pi}{x_{n+1+i}}$ and $\Delta_n=\Pi$. Define $\delta_{ij}=\Delta_i\Delta_j$. The following sets are linearly independent:

\begin{enumerate}
    \item $\mathcal{A}(n)=\{\delta_{ij}\mid 0\leq i\leq j\leq n\}$.
    \item     $$\mathcal{B}(n):=\{x_{n+1}\delta_{ij}\mid i,j\neq 0\}\cup \{x_{n+2}\delta_{ij}\mid i,j\neq 1\}\cup \dots $$$$\cup \{x_{2n}\delta_{ij}\mid i,j\neq n-1 \}.$$
\end{enumerate}
\end{lema}

\begin{proof}
We proceed by induction on $n\geq 3$. 
The case $n=3$ is covered by Lemma \ref{lema3}.

Assume that the result holds for $n=k$. For $n=k+1$, suppose that we have $\sum_{0\leq i\leq j\leq k+1}\alpha_{ij}\delta_{ij}=0$.
Observe that  
for all $ j>0$, $\delta_{0j}\in \langle x_{k+2}\rangle$,  and for all $i>0$, $\delta_{ij}\in \langle x_{k+2}^2\rangle$, so we can write 
\begin{equation}\label{deltas}
    \sum_{0\leq i\leq j\leq k+1}\alpha_{ij}\delta_{ij}=\alpha_{00}\delta_{00}+x_{k+2}A_1+x_{k+2}^2A_2=0,
\end{equation}

where $A_1=\sum_{1\leq j} \alpha_{0j}\frac{\delta_{0j}}{x_{k+2}}$ and $A_2=\sum_{1\leq i\leq j} \alpha_{ij}\frac{\delta_{ij}}{x_{k+2}^2}$.

From (\ref{deltas}), we have that  
$\alpha_{00}\delta_{00}\in \langle x_{k+2}\rangle$, that implies $\alpha_{00}=0$.
So we have $A_1+x_{k+2}A_2=0$, therefore $A_1\in \langle x_{k+2}\rangle$. We can factor
$$A_1=\Delta_0(\alpha_{01}\frac{\Delta_1}{x_{k+2}}+\dots  +\alpha_{0k+1}\frac{\Delta_{k+1}}{x_{k+2}}),$$ and deduce that $(\alpha_{01}\frac{\Delta_1}{x_{k+2}}+\dots  +\alpha_{0k+1}\frac{\Delta_{k+1}}{x_{k+2}})\in \langle x_{k+2}\rangle$.

Observe that $\frac{\Delta_j}{x_{k+2}}\in \langle x_{k+3}\rangle$, except for $j=1$, then we have  \[0\equiv \alpha_{01}\frac{\Delta_1}{x_{k+2}}+\dots  +\alpha_{0k+1}\frac{\Delta_{k+1}}{x_{k+2}}\equiv \alpha_{01}\frac{\Delta_1}{x_{k+2}}\text{ mod } \langle x_{k+2},x_{k+3}\rangle\] so $\alpha_{01}=0$. Analogously we obtain $\alpha_{0j}=0$ for all $j=1\dots k+1$.

Returning to (\ref{deltas}), we have 
$A_2=\sum_{1\leq i\leq j} \alpha_{ij}\frac{\delta_{ij}}{x_{k+2}^2}=0$. 

Observe that $\frac{\delta_{ij}}{x_{k+2}^2}\equiv \frac{\Delta_i}{x_{k+2}} \frac{\Delta_j}{x_{k+2}}\text{ mod } \langle x_{k+2}\rangle$. Thus $\frac{\delta_{ij}}{x_{k+2}^2}
\equiv \hat{\delta}_{ij}\text{ mod } \langle x_{k+2}\rangle$, where $\hat{\delta}_{ij}$ denotes the corresponding deltas for $n=k$, so by inductive hypothesis, they are linearly independent and $\alpha_{ij}=0$ for $1\leq i\leq j\leq k+1$.

The proof for $\mathcal{B}(n)$, is the same as for $n=3$.
Observe that 
\begin{equation}\label{xn+1}
    x_{n+1}\delta_{ij}\in \langle \Delta_1,\dots ,\Delta_n\rangle^2=x_{n+1}^3\langle \frac{\Pi}{x_{n+1}x_{n+2}}L,\dots, \frac{\Pi}{x_{n+1}x_{2n}}L ,\frac{\Pi}{x_{n+1}}\rangle^2
\end{equation} for all $i,j\neq 0$

Analogously, for $i,j\neq l$ we have 
\[x_{n+1+l}\delta_{ij}\in \langle \Delta_0,\dots,\hat{\Delta_l},\dots  ,\Delta_n\rangle^2=\]
\[x_{n+1+l}^3\langle \frac{\Pi}{x_{n+1+l}x_{n+1}}L,\dots ,\widehat{\frac{\Pi}{x_{n+1+l}^2}L} ,\dots, \frac{\Pi}{x_{n+1+l}x_{2n}}L,\frac{\Pi}{x_{n+1+l}}\rangle^2\]

 If  there was a zero linear combination of elements of $\mathcal{B}(n)$, we can write it in the following way:
 \begin{equation}\label{equationx2n}
     x_{n+1}^3A_1+x_{n+2}^3A_2+\dots +x_{2n}^3A_n=0,
 \end{equation}

hence  $A_1\in \langle x_{n+2},\dots, x_{2n}\rangle^3$. 
Using (\ref{xn+1}) we conclude that  $A_1=\lambda_1 (\frac{\Pi}{x_{n+1}})^2$.
Analogously we find 
$A_l=\lambda_l(\frac{\Pi}{x_{n+l}})^2$, for all $1\leq l\leq n$.
Using equation (\ref{equationx2n}) we conclude that $\lambda_1=\dots=\lambda_n=0$, i.e. $A_1=\dots=A_n=0$.
On the other hand, by definition, 
$0=x_{n+l}^2A_l$ is a linear combination of $\delta_{ij}$, which we have proven are linearly independent, so all the coefficients in (\ref{equationx2n}) are zero.

\end{proof}

\section{\texorpdfstring{The Lefschetz locus in $\Gor(1,n,n,1)$ for $n\leq 7$}{}}   
In this section, we gather all the previous information to study the cases of algebras in  $\Gor(1,n,n,1)$, for $n\leq 7$. 
We compute explicitly the dimension and degrees of the minimal and maximal family in $\P^{5}$ and $\P^{6}$ and we describe the locus of algebras failing SLP in this two particular cases.

\subsection{\texorpdfstring{The Lefschetz locus in $\Gor(1,6,6,1)$}{}}

As we have already seen, any Artinian Gorenstein algebra in $\Gor(1,6,6,1)$  corresponds to a homogeneous polynomial $f \in \K[x_0,\dots,x_5]_{3}$ (that is not a cone). In other words, $\Gor(1,6,6,1)=\P(S_3)\setminus \mathcal{C}=\P^{55}\setminus \mathcal{C}$. Moreover, by Remark \ref{Ns}, we know that  the algebra fails to have SLP if and only if $f\in \mathbb{X}_{min}$.
\begin{prop}
The locus in $\Gor(1,6,6,1)$ of algebras satisfying SLP is $\Gor(1,6,6,1) \setminus \mathbb{X}_{min}$.
\flushright{$\square$}
\end{prop}


Using the results  in Theorem \ref{thmMin} we obtain 

\begin{thm}\label{TheoremP5}
The locus $\mathbb{X}_{min}\subset \P^{55}$ is a rational irreducible projective variety of dimension $29$ and degree $51847992$. Moreover,  $\mathcal{H}_{red}=\mathcal{C}\cup \mathbb{X}_{min}$.

\end{thm}

\begin{proof}
By  Theorem \ref{thmMin}  we obtain $\dim(\P(\mathcal{E}))=29$.
The degree is given by $s_{11}(\mathcal{E})$, and we compute it using the Scripts in \S\ref{scripts} for $N=5$: $s_{11}(\mathcal{E})=51847992$.
\end{proof}

In the same way, applying Proposition \ref{minwithcones} we have:
\begin{prop}
In $\P^5$,  the intersection of the minimal family with cones  $\mathbb{X}_{min}\cap \mathcal{C}$ is a divisor in $\mathbb{X}$ of degree $98048160$. 
    
\end{prop}

\subsection{\texorpdfstring{The Lefschetz locus in $\Gor(1,7,7,1)$}{}}

 $\Gor(1,7,7,1)$ can be identified with cubics (non cones) in $\P^6$ i.e. $\P(S_3)\setminus \mathcal{C}=\P^{83}\setminus \mathcal{C}$. By Remark \ref{Ns}, we know that  $A_{f}$  fails to have SLP if and only if $f\in \mathbb{X}_{min}$ or $f\in \mathbb{X}_{max}$.

\begin{prop}\label{relation}
The Lefschetz locus in $\Gor(1,7,7,1)$ is $\Gor(1,7,7,1) \setminus (\mathcal{H}\setminus \mathcal{C})$. Furthermore $\overline{\mathcal{H}\setminus \mathcal{C}}=\mathbb{X}_{min}\cup \mathbb{X}_{max}$, 
where  $\mathbb{X}_{min}$ and $\mathbb{X}_{max}$ are the irreducible components of the reduced scheme of $\overline{\mathcal{H}\setminus \mathcal{C}}$.

\end{prop}

\begin{proof}
    To prove that $\mathbb{X}_{min}$ and $\mathbb{X}_{max}$ are  the irreducible components of the reduced scheme of $\overline{\mathcal{H}\setminus \mathcal{C}}$, it remains to prove that $\mathbb{X}_{min}$ isn't in the closure of $\mathbb{X}_{max}$. This is so because a generic element in $\mathbb{X}_{min}$ has a unique $2$-plane in its singular set, so it can't be the limit of elements of $\mathbb{X}_{max}$, ( all of them having  a $3-$plane in its singular sets).  
\end{proof}

The figure is as follows:




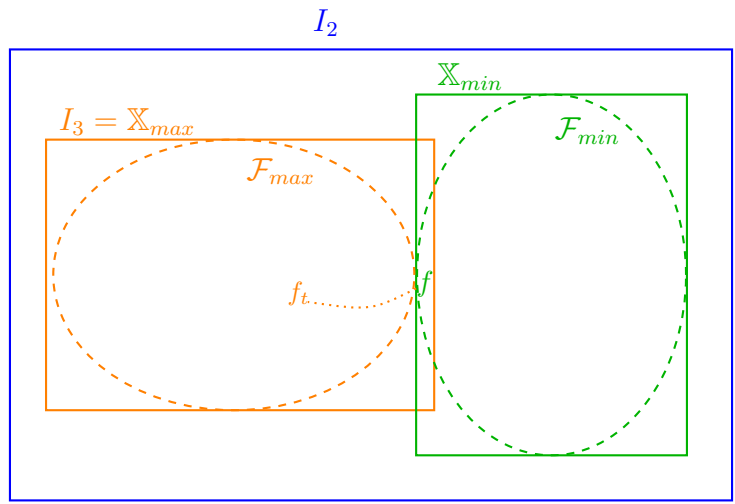
\begin{center}
\begin{tikzpicture}[scale=1.2]

\draw[blue, thick] (1,0) rectangle (9,5);
\node[blue] at (4.5,5.3) {\large $I_2$};

\draw[orange, thick] (1.4,1) rectangle (5.7,4);
\node[orange] at (2.3,4.2) {\large $I_3=\mathbb{X}_{max}$};

\draw[green!70!black, thick] (5.5,0.5) rectangle (8.5,4.5);
\node[green!70!black] at (6.1,4.7) {\large $\mathbb{X}_{min}$};

\draw[orange, dashed, thick] (3.48,2.5) ellipse (2 and 1.5);
\node[orange] at (4,3.65) {\large $\mathcal{F}_{max}$};

\draw[green!70!black, dashed, thick] (7,2.5) ellipse (1.49 and 2);
\node[green!70!black] at (7.4,4.1) {\large $\mathcal{F}_{min}$};

\node[green!70!black, font=\bfseries] at (5.6,2.4) {$f$};
\draw[orange, dotted, thick] (5.5,2.35) .. controls (5.0,2.1) .. (4.3,2.2);
\node[orange, font=\small] at (4.2,2.3) {$f_t$};

\end{tikzpicture}

\end{center}


\vspace{.5cm}
Where $I_k$ stands for the subvariety of cubics with a $k$-plane in their singular set, and  $\mathcal{F}_{min}$ (respectively $\mathcal{F}_{max}$) stands for the minimal family (maximal family, respectively).
The cubic \textcolor{green!70!black}{$f$} is the cubic of the remark \ref{rmkI^2}.

Using the results  in Theorem \ref{thmMin} we obtain 

\begin{thm}\label{theoremP6min}
The locus $\mathbb{X}_{min}$  is a rational projective irreducible variety of $\P^{83}$ of  dimension $44$ and degree $229416381544$.
\end{thm}

\begin{proof}
From Theorem \ref{thmMin}, we obtain the dimension of the family. The degree is computed  using the Script in \ref{scriptminN} with $N=6$.
\end{proof}

Applying Proposition \ref{minwithcones} we have:
\begin{prop}
In $\P^6$,  the intersection of the minimal family with cones  $\mathbb{X}_{min}\cap \mathcal{C}_3$ is a divisor in $\mathbb{X}_{min}$ of degree $378294450492$. 
    
\end{prop}

Regarding the maximal family, we have the following results.

\begin{thm}\label{degreeofmaxP6}
The locus $\mathbb{X}_{max}\subset \P^{83}$  is a rational projective irreducible variety of dimension $45$ and degree $5792937080$.
\end{thm}

\begin{proof}
By Theorem \ref{thmMax} the degree is 
 computed using the Script in \ref{script4} for $N=6$.  
\end{proof}

\begin{prop}
In $\P^6$,  the intersection of the maximal family with cones  $\mathbb{X}_{max}\cap \mathcal{C}_2$ has codimension $3$ in $\mathbb{X}_{max}$ and  degree $51258091892$.
\begin{proof}
    Apply Proposition \ref{maximalwithcones}.

\end{proof}    
\end{prop}



\appendix

\section{}\label{scripts}
\subsection{\texorpdfstring{Minimal family}{}}

\label{scriptminN}
{\small

\begin{verbatim}
loadPackage "Schubert2"
--choose your N
N= ;
-- Grassmannian of planes in N-space.
G1=flagBundle ({3,N-2}); 
-- names the sub and quotient bundles on G1
(Q1,Tau1)=G1.Bundles;

-- define F=grass(2,Tau1), the quotient has rank 2
F=flagBundle ({N-4,2},Tau1); 
-- names the sub and quotient bundles on F
(Q2,Tau2) = F.Bundles ;

--Define E1 and E2, such that E=E1+E2 
E1=(symmetricPower(2,Tau2))*Q1;
E2=symmetricPower(3,Tau1);

--compute the dimF-Segre class of the proof of the Theorem.
integral (segre (5*(N-2)-4,E1+E2)) 




\end{verbatim}
}

\subsection{\texorpdfstring{ Minimal family intersection with cones}{}}

\label{script1conesminimal}
{\small

\begin{verbatim}

loadPackage "Schubert2"
--give a value for N
N=; 

--define grass(N,N+1), 
--the quociente Tau0=Q0 has rk N.
G0=flagBundle ({1,N}); 
(S0,Q0) = G0.Bundles;

-- define  grass(N-2,Tau0),
--the  quocient Tau1=Q1 has rk  N-2.
F1=flagBundle ({2,N-2},Q0); 
(S1,Q1) = F1.Bundles;

-- define grass(2,Tau1), 
--the quotient  Tau2=Q2 has rk  2.
F2=flagBundle ({N-4,2},Q1); 
(S2,Q2) = F2.Bundles;


E1=(symmetricPower(2,Q2))*S1;
E2=symmetricPower(3,Q1);
rank (E1+E2);
integral(segre (dim F2,E1+E2))



\end{verbatim}
}

\subsection{\texorpdfstring{ Maximal family. $N=2k$}{}}

\label{script4}
{\small

\begin{verbatim}
loadPackage "Schubert2"
--choose your N
N=;
N1=N+1;
k=substitute(N/2,ZZ);
-- Grassmannian of k planes in N-space.
G=flagBundle({k+1,k})
-- names the sub and quotient bundles on G
(S,Q)=G.Bundles;

R=dual (Q);
A=symmetricPower(2,R);
B=A^N1;
C=exteriorPower(2,R)*R;
D=exteriorPower(3,R);

E=B+D-C

--Computes the dimG-Segre classes 
integral(segre(dim G,E))




\end{verbatim}
}

\subsection{\texorpdfstring{Maximal family intersection with cones}{}}

\label{script1conesmaximal}
{\small

\begin{verbatim}
N=;
loadPackage "Schubert2"
k=substitute(N/2,ZZ);
G0=flagBundle({1,N});-- Grassmannian of 0-planes in 2k+1-space
(Q0,Tau0)=G0.Bundles;

G=flagBundle({k,k}, Tau0);
(Q,Tau1)=G.Bundles

R=Tau1;
A=symmetricPower(2,R);
B=A*Tau0;
C=exteriorPower(2,R)*R;
D=exteriorPower(3,R);
E=B+D-C;

--Computes the classes 
integral(segre(dim G,E))


\end{verbatim}
}

{\bf Acknowledgments}.
We wish to thank F. Russo for inspiring conversations on the subject. The first author was financed in part by the Coordena\c{c}\~{a}o de Aperfei\c{c}oamento de Pessoal de N\'{i}vel Superior - Brasil (CAPES) - Finance Code 001. The last author was partially supported by CNPQ. 
Ferrer acknowledges support from CNPq (Grant number 408687/2023-1). 


\bibliography{v1}
\bibliographystyle{abbrv}

\end{document}